\documentclass[11pt,a4paper]{amsart}

\usepackage{amssymb}
\usepackage{amscd}
\usepackage{amsmath,amsfonts,stmaryrd, graphicx, graphics}
\usepackage[usenames,dvipsnames,svgnames]{xcolor}
\usepackage[hypcap=false]{caption}

\usepackage{silence}
\usepackage[colorlinks, urlcolor=Navy, linkcolor=Navy, citecolor=Navy]{hyperref}
\usepackage{ytableau}
\usepackage{microtype}
\usepackage{rotating}
\usepackage{lscape}
\usepackage{afterpage}

\usepackage[abbrev,nobysame]{amsrefs}
\renewcommand{\MR}[1]{}

\usepackage[all]{xy}

\usepackage{tikz}
\usepackage{tikz-cd}
\usetikzlibrary{arrows,shapes,chains,matrix,positioning,scopes,cd,patterns}

\theoremstyle{plain}
\newtheorem{thm}{Theorem}[section]
\newtheorem{lem}[thm]{Lemma}

\newtheorem{pro}[thm]{Proposition}
\newtheorem{cor}[thm]{Corollary}
\newtheorem{thm*}{Theorem}
\theoremstyle{definition}
\newtheorem{dfn}[thm]{Definition}
\newtheorem{exa}[thm]{Example}
\newtheorem{rem}[thm]{Remark}

\numberwithin{equation}{section}

\DeclareMathOperator{\Hom}{Hom}

\DeclareMathOperator{\End}{End}
\DeclareMathOperator{\Ext}{Ext}

\DeclareMathOperator{\modu}{mod}

\newcommand{\lmod}[1]{\text{$#1$-$\modu$}}

\newcommand{\fb}{fb}

\newcommand{\Ga}{\Gamma}

\DeclareMathOperator{\im}{im}

\DeclareMathOperator{\pd}{pd}

\newcommand{\oTo}{\xymatrix{ \ar@{^{(}->}[r]|{\mathbf{O}}& }} 
\newcommand{\cTo}{\xymatrix{ \ar@{^{(}->}[r]|{\mathbf{|}}& }} 
\newcommand{\coTo}{\xymatrix{ \ar@{^{(}->}[r]|{\mathbf{O}}|{\mathbf{|}}& }} 

\DeclareMathOperator{\Ker}{ker}
\DeclareMathOperator{\coKer}{coker}
\DeclareMathOperator{\Bild}{im}

\DeclareMathOperator{\soc}{soc}
\DeclareMathOperator{\Top}{top}

\DeclareMathOperator{\add}{add}

\newcommand{\kdual}{\mathrm{D}}

\newcommand{\La}{\Lambda}

\newcommand{\N}{\mathbb{N}}

\newcommand{\Z}{\mathbb{Z}}

\newcommand{\mcG}{\mathcal{G}}

\DeclareMathOperator{\gen}{gen}
\DeclareMathOperator{\cogen}{cogen}

 \newcommand{\cohook}{{\rm cohook}}

\usepackage{todonotes}

\renewcommand{\epsilon}{\varepsilon}

\renewcommand{\tilde}{\widetilde}

\begin{document}
\title{Combinatorics of faithfully balanced modules}
\author{William Crawley-Boevey}
\author{Biao Ma}
\author{Baptiste Rognerud}
\author{Julia Sauter}

\address{William Crawley-Boevey, Julia Sauter: Fakult\"at f\"ur Mathematik, Universit\"at Bielefeld, 33501 Bielefeld, Germany}

\email{wcrawley, jsauter@math.uni-bielefeld.de}

\address{Biao Ma: Department of Mathematics, Nanjing University, 22 Hankou Road, Nanjing 210093, People's Republic of China}


\email{biao.ma@math.uni-bielefeld.de}

\address{Baptiste Rognerud: Universit\'e Paris Diderot, Institut de mathématiques de Jussieu - Paris Rive Gauche, 75013 Paris, France}

\email{baptiste.rognerud@imj-prg.fr}
\subjclass[2010]{Primary 16D80,16G10}


\keywords{faithfully balanced module, double centralizer property, tilting module, Nakayama algebra, binary tree}

\thanks{The first, third and fourth authors have been supported by the Alexander von Humboldt Foundation 
in the framework of an Alexander von Humboldt Professorship 
endowed by the German Federal Ministry of Education and Research. The second author is supported by the China Scholarship Council.}

\begin{abstract}
We study and classify faithfully balanced modules for the algebra of triangular $n$ by $n$ matrices and more generally for Nakayama algebras. The theory extends known results about tilting modules, which are classified by binary trees, and counted with the Catalan numbers. The number of faithfully balanced modules is a $2$-factorial number. Among them are $n!$ modules with $n$ indecomposable summands, which can be classified by interleaved binary trees or by increasing binary trees. 
\end{abstract}
\maketitle

\section{Introduction}
We consider the category $\lmod{\Lambda}$ of finitely generated left $\Lambda$-modules, where $\Lambda$ is a finite-dimensional algebra over a field $K$, or more generally an artin algebra. Recall that a module $M$ is said to be \emph{balanced}, or to have the \emph{double centralizer property} if the natural map $\Lambda\to \End_E(M)$ is surjective, where $E=\End_\Lambda(M)$, and it is said to be \emph{faithfully balanced} if the natural map is bijective, or equivalently if $M$ is faithful and balanced.

Balanced and faithfully balanced modules appear in various places in the literature on ring theory, such as Schur-Weyl duality (see for example \cite{KSX01}), and Thrall's notion of a QF-1 algebra \cite{Thrall48}. The main known examples of faithfully balanced modules are faithful modules for a self-injective algebra, and more generally generators and cogenerators for any algebra, and tilting modules and cotilting modules. For more examples see \cite{MS19}.

In general the behaviour of faithfully balanced modules is rather mysterious. We shall illustrate this by studying these modules for the algebra $\Lambda_n$ of $n\times n$ lower triangular matrices over $K$, or equivalently the path algebra of the linearly oriented $A_n$ quiver
\[
1\to 2\to\dots \to n.
\]
The indecomposable modules for $\Lambda_n$ are indexed by the set $I_n = \{ (i,j) : 1\le i\le j\le n\}$, which we display as the blocks of a Young diagram of staircase shape
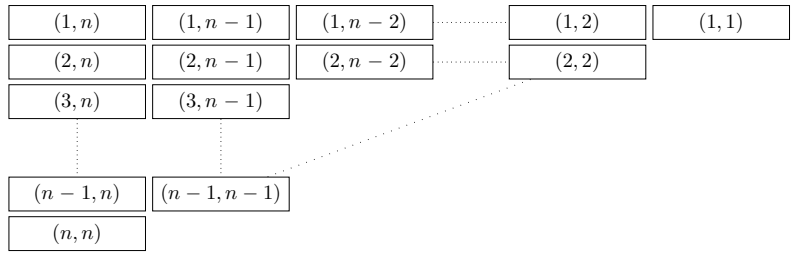
\[
\scalebox{0.7}{
\tikzstyle{block} = [rectangle, draw, fill=white!20,
    text width=6em, text centered, minimum height=1em, node distance=2.7cm]
\tikzstyle{line} = [draw, very thick, color=black!50, -latex']
\begin{tikzpicture}[scale=4, node distance = 0cm]
    \node [block] (1n) {$(1,n)$};
    \node [block, right of=1n] (1n-1) {$ (1,n-1)$};
     \node [block, right of=1n-1] (1n-2) {$ (1,n-2)$};
     \node [block, right of=1n-2,node distance=4cm] (12) {$(1,2)$};
     \node [block, right of=12] (11) {$(1,1)$};
    \node [block, below of=1n, node distance=0.75cm] (2n) {$(2,n)$};
    \node [block, right of=2n] (2n-1) {$ (2,n-1)$};
     \node [block, right of=2n-1] (2n-2) {$ (2,n-2)$};
     \node [block, right of=2n-2,node distance=4cm] (22) {$(2,2)$};
      \node [block, below of=2n, node distance=0.75cm] (3n) {$(3,n)$};
    \node [block, right of=3n] (3n-1) {$ (3,n-1)$};
     \node [block, below of=3n, node distance=1.75cm] (n-1n) {$ (n-1,n)$};
    \node [block, right of=n-1n] (n-1n-1) {$ (n-1,n-1)$};
      \node [block, below of=n-1n, node distance=0.75cm] (nn) {$(n,n)$};     
     \draw [dotted](1n-2) -- (12) ;
     \draw [dotted](2n-2) -- (22) ;
     \draw [dotted](3n) -- (n-1n) ;
     \draw [dotted](3n-1) -- (n-1n-1) ;
     \draw [loosely dotted](n-1n-1) -- (22) ;
\end{tikzpicture}
}
\]
The element $(i,j)$ corresponds to the module $M_{ij}$ with top and socle the simple modules $S[i]$ and $S[j]$. The left hand column is the indecomposable projective modules, the top row is the indecomposable injective modules and the modules $M_{ii}$ are the simple modules $S[i]$. The Auslander-Reiten quiver is the same picture, with irreducible maps going vertically and to the right, and the Auslander-Reiten translation $\tau = DTr$ takes each module $M_{ij}$ with $j<n$ to $M_{i+1,j+1}$.
By a \emph{leaf} we mean an element of the set $L = \{ (1,0),(2,1),\dots,(n+1,n) \}$. We define \emph{cohooks} for $(i,j) \in I_n$ and \emph{virtual cohooks} for $(i,j)\in L$ by the formula
\[
\cohook(i,j) = \{ M_{kj} : 1\le k<i \} \cup \{ M_{i\ell} : n\ge \ell > j \}
\]

In section \ref{s:fbcond} we prove the following theorem, along with its generalization to Nakayama algebras and a version for balanced modules.

\begin{thm}
\label{t:fban}
A $\Lambda_n$-module $M$ is faithfully balanced if and only if it satisfies the following conditions:
\begin{itemize}
\item[(FB0)]$M_{1n}$ is a summand of $M$;
\item[(FB1)]if $M_{ij}$ is a summand of $M$, $(i,j)\neq(1,n)$, then $\cohook(i,j)$ contains a summand of $M$; and
\item[(FB2)]every virtual cohook contains a summand of $M$.
\end{itemize}
\end{thm}

For example the faithfully balanced modules for $\Lambda_3$ are given by taking copies of the indecomposable modules corresponding to the the black boxes $\blacksquare$ in one of the diagrams in Figure~\ref{fig_fbgencats}, together with an arbitrary subset of the shaded boxes $\boxtimes$. A module is \emph{basic} if its indecomposable summands occur with multiplicity one. The diagrams show the $8+4+2+2+1+2+2=21$ basic faithfully balanced modules for $\Lambda_3$.

Given an algebra $\Lambda$ and a module $M$, we write $\add(M)$ for the full subcategory of $\lmod{\Lambda}$ consisting of the direct summands of direct sums of copies of $M$, $\gen(M)$ for the category of modules generated by $M$, so quotients of a direct sum of copies of $M$, and $\cogen(M)$ for the category of modules cogenerated by $M$, so embeddable in a direct sum of copies of $M$. Recall from Pressland and Sauter \cite[Definition 2.10]{PS18}, that if $M$ is a $\Lambda$-module, then $\gen_1(M)$ is the category of modules $X$ such that there is an exact sequence $M''\to M'\to X\to 0$ with $M',M''\in\add(M)$ and the sequence
\[
\Hom(M,M'')\to \Hom(M,M')\to \Hom(M,X)\to 0
\]
exact, and $\cogen^1(M)$ is the category of modules $X$ such that there is an exact sequence $0\to X\to M'\to M''$ with $M',M''\in\add(M)$ and the sequence
\[
\Hom(M'',M)\to \Hom(M',M)\to \Hom(X,M)\to 0
\]
exact. These are full subcategories of $\lmod{\Lambda}$, closed under direct sums and summands. It is known that $M$ is faithfully balanced if and only if the projective modules are all in $\cogen^1(M)$ or equivalently the injective modules are all in $\gen_1(M)$ (see Lemma \ref{l:fbchar}). It follows that the property of $M$ being faithfully balanced only depends on $\add(M)$, and so one may assume that $M$ is basic. By a (faithfully balanced) \emph{$\gen_1$-category} or \emph{$\cogen^1$-category} we mean a subcategory of $\lmod{\Lambda}$ of the form $\gen_1(M)$ or $\cogen^1(M)$ respectively, where $M$ is some (faithfully balanced) module. We say that a $\Lambda$-module $M$ is \emph{$\gen_1$-critical} if any proper summand $N$ of $M$ has $\gen_1(N) \neq \gen_1(M)$; similarly for $\cogen^1$-critical. In section \ref{s:crit} we prove the following.

\begin{thm}
\label{t:gencrit}
For the algebra $\Lambda_n$, or for any representation-directed algebra $\Lambda$, any $\gen_1$-category $\mcG$ contains a $\gen_1$-critical module $M$ with $\gen_1(M)=\mcG$, which is unique up to isomorphism. For any module $L$, we have $\gen_1(L) = \mcG$ if and only if $\add(M)\subseteq \add(L) \subseteq \mcG$.
\end{thm}

\begin{figure}[ht]
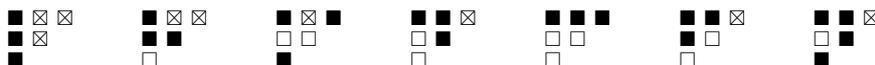

\[
\begin{smallmatrix}
\blacksquare & \boxtimes  & \boxtimes \\[1pt]
\blacksquare & \boxtimes   \\[1pt]
\blacksquare 
\end{smallmatrix}
\qquad
\begin{smallmatrix}
\blacksquare & \boxtimes & \boxtimes  \\[1pt]
\blacksquare & \blacksquare  \\[1pt]
\square  
\end{smallmatrix}
\qquad
\begin{smallmatrix}
\blacksquare & \boxtimes  & \blacksquare \\[1pt]
\square  & \square  \\[1pt]
\blacksquare  
\end{smallmatrix}
\qquad
\begin{smallmatrix}
\blacksquare & \blacksquare & \boxtimes \\[1pt]
\square  & \blacksquare   \\[1pt]
\square  
\end{smallmatrix}
\qquad
\begin{smallmatrix}
\blacksquare & \blacksquare & \blacksquare \\[1pt]
\square  & \square  \\[1pt]
\square   
\end{smallmatrix}
\qquad
\begin{smallmatrix}
\blacksquare & \blacksquare  & \boxtimes  \\[1pt]
\blacksquare & \square  \\[1pt]
\square
\end{smallmatrix}
\qquad
\begin{smallmatrix}
\blacksquare & \blacksquare  & \boxtimes  \\[1pt]
\square & \blacksquare  \\[1pt]
\blacksquare 
\end{smallmatrix}
\]
\caption{The seven faithfully balanced $\gen_1$-categories for~$\Lambda_3$. The black boxes $\blacksquare$ show the summands of the $\gen_1$-critical module, and together with the shaded boxes $\boxtimes$ they show the category $\gen_1(M)$.}\label{fig_fbgencats}
\end{figure}

Figure~\ref{fig_fbgencats} shows the faithfully balanced $\gen_1$-categories for~$\Lambda_3$. 
We say that a module is \emph{minimal faithfully balanced} if it is faithfully balanced and any proper direct summand is not faithfully balanced. Clearly any minimal faithfully balanced module is $\gen_1$- and $\cogen^1$-critical. Any (generalized) tilting module $T$ is faithfully balanced, see \cite{W90} and \cite[Proposition 5]{W88}. In section~\ref{s:crit} we prove the following.

\begin{thm}
\label{t:tilting}
If $T$ is a basic classical tilting module for an artin algebra $\Lambda$, i.e.\ $T$ has projective dimension $\le 1$, then $T$ is $\gen_1$-critical. If in addition $\Lambda$ is hereditary, then $T$ is minimal faithfully balanced. 
\end{thm}

It follows that any $\tau$-tilting module is $\gen_1$-critical and balanced. In Figure~\ref{fig_fbgencats}, the first five $\gen_1$-critical modules are tilting modules. These and the sixth module are minimal faithfully balanced. The last $\gen_1$-critical module is not minimal faithfully balanced. Note that although all minimal faithfully balanced modules for $\Lambda_3$ have 3 indecomposable summands, the module
\[
\begin{smallmatrix}
\blacksquare & \square  & \blacksquare & \square  \\[1pt]
\square  & \square & \blacksquare  \\[1pt]
\blacksquare & \blacksquare  \\[1pt]
\square 
\end{smallmatrix}
\]
is a minimal faithfully balanced for $\Lambda_4$, but it has more than 4 indecomposable summands. To count faithfully balanced modules for $\Lambda_n$ we prove the following in section \ref{s:fbcount}.

\begin{thm}
\label{t:fbcount}
In the expansion of the polynomial
\[
h_n(x_1,\dots,x_n) = 
\prod_{r=1}^n \left( \prod_{s=1}^r (1+x_s) - 1 \right),
\]
the coefficient of the monomial $x_1^{t_1} \dots x_n^{t_n}$ is the number of basic faithfully balanced $\Lambda_n$-modules $M$ with $t_i$ indecomposable summands having top $S[i]$ (or equivalently in row $i$ of the Young diagram), for all $i$. 
\end{thm}

It follows that the number of basic faithfully balanced modules for $\Lambda_n$ is the 2-factorial number
\[
[n]_2!:=\prod_{i=1}^n (2^i -1).
\]
For example there are $(2-1)(2^2-1)(2^3-1)=21$ basic faithfully balanced $\Lambda_3$-modules. 
Also, any basic faithfully balanced module for $\Lambda_n$ has at least $n$ summands, and the number with exactly $n$ summands is $n!$. For comparison, note that the number of basic tilting modules for $\Lambda_n$ is the $n$th Catalan number, see \cites{Gab81,hille_vol_tilting,Rin16}. In section~\ref{s:quadratic} we use Theorem~\ref{t:fbcount} to count faithfully balanced modules over certain quadratic Nakayama algebras. One should remark that the number of summands of a faithfully balanced module for a more general algebra may be less than the number of isomorphism classes of simple modules. For example the direct sum of the indecomposable projective-injective modules for an Auslander algebra is faithfully balanced. This can also happen for some non-linear orientations of the path algebra of $A_n$ for $n\geq 5$.

In section \ref{section_interleaved}, we investigate the combinatorics of faithfully balanced modules with exactly $n$ indecomposable summands. Recall that an increasing binary tree with $n$ vertices is a binary tree with a labelling of the vertices by the integers $1,\dots,n$, such that the label of any vertex is less that that of any of its children. See Definition~\ref{d:interleaved} for the notion of an `interleaved tree'.

\begin{thm}
\label{t:fbncomb}
Given $n$, there are explicit bijections between the following types of objects:
\begin{itemize}
\item[(i)]faithfully balanced modules for $\Lambda_n$ with exactly $n$ indecomposable summands;
\item[(ii)]interleaved trees with $n$ vertices;
\item[(iii)]increasing binary trees with $n$ vertices;
\item[(iv)]functions $f:\{1,\dots,n\}\to \{1,\dots,n\}$ which are \emph{self-bounded}, meaning that $f(i)\le i$ for all $i$.
\end{itemize}
These restrict to bijections between basic tilting modules; binary trees; well-ordered increasing binary trees and non-decreasing self-bounded functions.
\end{thm}

Writing $\fb (n)$ for the set of faithfully balanced modules with exactly $n$ indecomposable summands, we also prove that there is a simple bijection between $\fb (n)$ and the set of tree-like tableaux of size $n$ in the sense of \cite{tree-like}.

In section \ref{s:poset} we study the poset structure $\unlhd$ on $\fb (n)$ given by $N \unlhd M$ if and only if $\cogen (N) \subseteq \cogen (M)$ and $\gen (N) \supseteq \gen (M)$. 
See Figure~\ref{fig_intro} for the cases $n=3$ and 4.
We prove the following.

\begin{figure}[ht]
\[
\raisebox{0cm}{
\scalebox{0.5}{
\xymatrix{  \; &{\begin{smallmatrix}
\blacksquare & \square  & \square  \\[1pt]
\blacksquare  & \square\\[1pt]
\blacksquare 
\end{smallmatrix}} \ar[ddr]\ar[dl] & && \\
 {\begin{smallmatrix}
\blacksquare & \square  & \square  \\[1pt]
\blacksquare  & \blacksquare\\[1pt]
\square 
\end{smallmatrix}} \ar[d] && && \\
  {\begin{smallmatrix}
\blacksquare & \blacksquare  & \square  \\[1pt]
\blacksquare  & \square\\[1pt]
\square 
\end{smallmatrix}}\ar[d] && {\begin{smallmatrix}
\blacksquare & \square  & \blacksquare  \\[1pt]
\square  & \square\\[1pt]
\blacksquare 
\end{smallmatrix}} \ar[ddl] && \\
 {\begin{smallmatrix}
\blacksquare & \blacksquare  & \square  \\[1pt]
\square  & \blacksquare\\[1pt]
\square 
\end{smallmatrix}}\ar[dr]&& &&\\
 &{\begin{smallmatrix}
\blacksquare & \blacksquare  & \blacksquare  \\[1pt]
\square  & \square\\[1pt]
\square 
\end{smallmatrix}} & &&}}  } 
\raisebox{-4cm}{
\begin{tikzpicture}
\coordinate (1) at (1.414,.707,0);
\coordinate (2) at (1.414,0,.707);
\coordinate (3) at (1.414,-.707,0);
\coordinate (4) at (1.414,0,-.707);
\coordinate (5) at (-1.414,.707,0);
\coordinate[label=left:\scalebox{0.7}{$\begin{smallmatrix}
\blacksquare & \blacksquare  & \blacksquare & \blacksquare  \\[1pt]
\square  & \square & \square  \\[1pt]
\square & \square  \\[1pt]
\square 
\end{smallmatrix}$}] (6) at (-1.414,0,.707);
\coordinate(7) at (-1.414,-.707,0);
\coordinate (8) at (-1.414,0,-.707);
\draw (-1.414,0,.707) circle (3pt);
\coordinate (9) at (.707,1.414,0);
\coordinate (10) at (0,1.414,.707);
\coordinate (11) at (-.707,1.414,0);
\coordinate[label=above:\scalebox{0.7}{$\begin{smallmatrix}
\blacksquare & \square  & \square & \square  \\[1pt]
\blacksquare  & \square & \square  \\[1pt]
\blacksquare & \square  \\[1pt]
\blacksquare 
\end{smallmatrix}$}] (12) at (0,1.414,-.707);
\fill (0,1.414,-.707) circle (3pt);
\coordinate (13) at (.707,-1.414,0);
\coordinate (14) at (0,-1.414,.707);
\coordinate (15) at (-.707,-1.414,0);
\coordinate (16) at (0,-1.414,-.707);
\coordinate (17) at (.707,0,1.414);
\coordinate (18) at (0,.707,1.414);
\coordinate (19) at (-.707,0,1.414);
\coordinate (20) at (0,-.707,1.414);
\coordinate (21) at (.707,0,-1.414);
\coordinate (22) at (0,.707,-1.414);
\coordinate (23) at (-.707,0,-1.414);
\coordinate (24) at (0,-.707,-1.414);
\draw[black] (1.414,.707,0) -- (1.414,0,.707) -- (1.414,-.707,0) -- (1.414,0,-.707) -- cycle;
\draw[black,dashed] (-1.414,.707,0) -- (-1.414,0,.707) -- (-1.414,-.707,0) -- (-1.414,0,-.707) -- cycle;
\draw[black] (.707,1.414,0) -- (0,1.414,.707) -- (-.707,1.414,0) -- (0,1.414,-.707) -- cycle;
\draw[black,dashed] (.707,-1.414,0) -- (0,-1.414,.707) -- (-.707,-1.414,0) -- (0,-1.414,-.707) -- cycle;
\draw[black] (.707,0,1.414) -- (0,.707,1.414) -- (-.707,0,1.414) -- (0,-.707,1.414) -- cycle;
\draw[black,dashed] (.707,0,-1.414) -- (0,.707,-1.414) -- (-.707,0,-1.414) -- (0,-.707,-1.414) -- cycle;
\draw (1.414,.707,0) -- (.707,1.414,0)
  (1.414,0,.707) -- (.707,0,1.414)
  (0,1.414,.707) -- (0,.707,1.414)
  (1.414,-.707,0) -- (.707,-1.414,0)
  (-.707,0,1.414) -- (-1.414,0,.707)
  (0,-.707,1.414) -- (0,-1.414,.707)
  (-.707,1.414,0) -- (-1.414,.707,0)
  (-1.414,-.707,0) -- (-.707,-1.414,0);
   \draw[dotted] (8) -- (23);
   \draw[dotted] (12) -- (22);
    \draw[dotted] (21) -- (4);
     \draw[dotted] (24) -- (16);
     \draw[black] (20) -- (3);
       \draw[dotted] (21) -- (13);
        \draw[black] (5) -- (6);
          \draw[black] (6) -- (7);
          \draw[black] (15) -- (14);
          \draw[black] (14) -- (13);
\end{tikzpicture}}
\]
\caption{The Hasse diagram of $(\fb(3), \unlhd)$ and the graph of the Hasse diagram of $(\fb(4), \unlhd )$. For more details see Figures~\ref{fig_posetfb3} and \ref{fig_fb4}.}\label{fig_intro}
\end{figure}
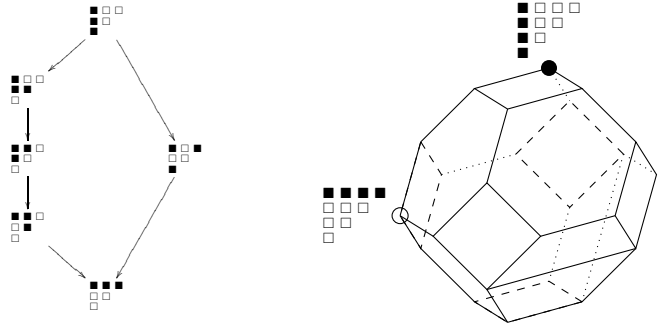

\begin{thm}\label{thm-lattice}
\begin{enumerate}
\item The poset $(\fb(n),\unlhd)$ is lattice.
\item The Tamari lattice is a sub-lattice of $(\fb(n),\unlhd)$.
\item The cover relations in $(\fb(n),\unlhd)$ are given by exchanging exactly one indecomposable summand.
\end{enumerate}
\end{thm}

Experiments were carried out using the GAP-package QPA \cite{qpa} and SageMath \cite{sage}.

\section{Characterizations of (faithfully) balanced modules}
In this section we consider finitely generated modules for an artin algebra $\Lambda$ and write $\Hom$ for $\Hom_\Lambda$. Recall that a morphism $f:X\to M'$ is called a \emph{left $\add(M)$-approximation} of $X$ if $M'\in\add(M)$ and any morphism $g: X\to M$ factors as $g = hf$ for some $h:M'\to M$. It is a \emph{minimal left approximation} if in addition $f$ is a \emph{left minimal} morphism, which means that for any $\phi\in\End(M')$, if $\phi f=f$ then $\phi$ is an automorphism, or equivalently that $\Bild\phi$ is not contained in a proper direct summand of $M'$, see \cite[\S1]{ASm1}. Dually there is the notion of a (minimal) right $\add(M)$-approximation. Minimal $\add(M)$-approximations exist, and are unique up to isomorphism, see \cite[Propositions 3.9, 4.2]{ASm1}. The combination of (a) and (b) is the case $k=1$ of \cite[Lemma 2.2]{MS19}.

\begin{lem}
\label{l:balX}
Let $X$ and $M$ be $\Lambda$-modules and let $E=\End(M)$. 

\noindent
(a) The following are equivalent:
\begin{itemize}
\item[(i)]the natural map $X\to \Hom_E(\Hom(X,M),M)$ is injective;
\item[(ii)]$X \in\cogen(M)$;
\item[(iii)]the minimal left $\add(M)$-approximation $\theta:X\to M'$ is injective.
\end{itemize}
(b) The following are equivalent:
\begin{itemize}
\item[(i)]the natural map $X\to \Hom_E(\Hom(X,M),M)$ is surjective;
\item[(ii)]there is a sequence $X\to M'\to M''$, exact in the middle and with $M',M''\in\add(M)$, such that the sequence 
$
\Hom(M'',M) \to \Hom(M',M) \to \Hom(X,M)\to 0
$
is exact:
\item[(iii)]the minimal left $\add(M)$-approximation $\theta:X\to M'$ 
has cokernel in $\cogen(M)$.
\end{itemize}
\end{lem}

\begin{proof}
Part (a) is straightforward, see for example \cite[Lemma VI.1.8]{ASS}; we prove (b). For (i) implies (iii), we consider the sequence $X\xrightarrow{\theta} M'\xrightarrow{\phi} M''$ where $\phi$ is the composition of $M'\to\coKer\theta$ and a left $\add(M)$-approximation of $\coKer\theta$. Then the sequence
\[
\Hom(M'',M)\to\Hom(M',M)\to\Hom(X,M)\to 0
\]
is exact. This gives a commutative diagram
\[
\begin{CD}
& & X @>>> M' @>>> M''  \\
& & @VfVV @VgVV @VhVV \\
0 @>>> \Hom_E(\Hom(X,M),M) @>>> \Hom_E(\Hom(M',M),M) @>>> \Hom_E(\Hom(M'',M),M)
\end{CD}
\] 
in which the bottom row is exact. Since $M'\in\add(M)$ it follows that $g$ is an isomorphism. By (i) the map $f$ is surjective. By diagram chasing the top row is exact, giving (iii).

(iii) implies (ii). 
One takes $\phi$ to be the composition of the map $M'\to\coKer\theta$ followed by a left $\add(M)$-approximation $\coKer\theta\to M''$ of $\coKer\theta$.

(ii) implies (i).
We have a commutative diagram as displayed above with exact rows.
Since $M',M''\in\add(M)$ the maps $g,h$ are isomorphisms. It follows that $f$ is a surjection.
\end{proof}

Considering the duals of $X$ and $M$ as $\Lambda^{op}$-modules (see \cite[\S II.3]{ARS}) gives the following.

\begin{lem}
\label{l:dualbalX}
Let $X$ and $M$ be $\Lambda$-modules and let $E=\End(M)$. 

\noindent
(a) The following are equivalent:
\begin{itemize}
\item[(i)]the natural map $\Hom(M,X)\otimes_E M\to X$ is surjective;
\item[(ii)]$X \in\gen(M)$;
\item[(iii)]the minimal right $\add(M)$-approximation $\theta:M'\to X$ is surjective.
\end{itemize}
(b) The following are equivalent:
\begin{itemize}
\item[(i)]the natural map $\Hom(M,X)\otimes_E M\to X$ is injective;
\item[(ii)]there is a sequence $M''\to M'\to X$, exact in the middle and with $M',M''\in\add(M)$, such that the sequence 
$
\Hom(M,M'') \to \Hom(M,M') \to \Hom(M,X)\to 0
$ 
is exact;
\item[(iii)]the minimal right $\add(M)$-approximation $\theta:M'\to X$ 
has kernel in $\gen(M)$.
\end{itemize}
\end{lem}

Using the additivity property of minimal approximations, one gets the following.

\begin{lem}
For a module $M$ the following are equivalent:
\begin{itemize}
\item[(i)]$M$ is balanced;
\item[(ii)]for every indecomposable projective module $P$, the minimal left $\add(M)$-approximation $\theta:P\to M'$ has cokernel in $\cogen(M)$;
\item[(iii)]for every indecomposable injective module $I$, the minimal right $\add(M)$-approximation $\theta:M'\to I$ has kernel in $\gen(M)$.
\end{itemize}
\end{lem}

As mentioned in the introduction, following Pressland and Sauter \cite{PS18}, we write $\gen_1(M)$ (respectively $\cogen^1(M)$) for the full subcatetegory of $\lmod{\Lambda}$ consisting of the modules $X$ satisfying the conditions in parts (a) and (b) of Lemma~\ref{l:dualbalX} (respectively Lemma~\ref{l:balX}). These subcategories are closed under direct sums and summands. The following consequence is already in \cite{PS18}.

\begin{lem}\label{l:fbchar}
For a module $M$, the following conditions are equivalent.
\begin{itemize}
\item[(i)]$M$ is faithfully balanced;
\item[(ii)]all projective $\Lambda$-modules are in $\cogen^1(M)$;
\item[(iii)]all injective $\Lambda$-modules are in $\gen_1(M)$.
\end{itemize}
\end{lem}

\section{Approximations for Nakayama algebras}
In this section $\Lambda$ is a Nakayama algebra, meaning that all indecomposable projective and injective modules are uniserial, see for example \cite[\S IV.2]{ARS}. It follows that any indecomposable module $X$ is uniserial, determined up to isomorphism by its length $\ell(X)$ and either its socle $\soc(X)$ or top $\Top(X)$.

We fix an indecomposable module $X$.

\begin{lem}\label{l:3.1}
Let $\phi:X\to U$ and $\phi':X\to U'$ be non-zero homomorphisms with $U,U'$ indecomposable. 
\begin{itemize}
\item[(i)]If $\theta\in\End(U)$ satisfies $\theta\phi=\phi$, then $\theta$ is invertible.
\item[(ii)]$\phi' = \theta\phi$ for some morphism $\theta:U\to U'$ $\Leftrightarrow$ $\ell(\Ker\phi) \le \ell(\Ker\phi')$ and $\ell(\coKer\phi) \le \ell(\coKer\phi')$.
\item[(iii)]$\phi' = \theta\phi$ for some isomorphism $\theta:U\to U'$ $\Leftrightarrow$ $\ell(\Ker\phi) = \ell(\Ker\phi')$ and $\ell(\coKer\phi) = \ell(\coKer\phi')$.
\end{itemize}
\end{lem}

\begin{proof}
(i) Since $U$ is indecomposable and $\phi$ is non-zero, $\phi:X\to U$ is left minimal.

(ii) If there is $\theta$, then trivially $\ell(\Ker \phi)\le \ell(\Ker \phi')$. Moreover since $\phi'\neq 0$, we have $\Bild \phi\not\subseteq \Ker\theta$, so since $U$ is uniserial, $\Ker\theta\subseteq \Bild \phi$. Thus $\theta^{-1}(\Bild \phi') = \Ker\theta+\Bild\phi = \Bild\phi$, so $\theta$ induces an injection from $\coKer\phi$ to $\coKer\phi'$, giving the other inequality.

Conversely if the inequalities hold, then $\Bild \phi$ and $\Bild \phi'$ are both quotients of the uniserial module $X$, so the inequality $\ell(\Ker \phi)\le \ell(\Ker \phi')$ ensures the existence of a surjective map $\alpha:\Bild \phi\to\Bild \phi'$ with $\phi' = \alpha\phi$. Taking the injective envelopes $I$ and $I'$ of $\Bild\phi$ and $\Bild\phi'$, the map $\alpha$ extends to a map $\beta:I\to I'$. Now $I$ and $I'$ are indecomposable, hence uniserial, since the modules $\Bild\phi$ and $\Bild\phi'$ have simple socle. Moreover $U$ embeds in $I$ and $U'$ in $I'$. Now $\Ker\beta \cap \Bild \phi = \Ker\alpha$, and $\Bild\phi\not\subseteq \Ker\beta$, so since $I$ is uniserial,
$\Ker\beta\subseteq\Bild\phi$, so $\Ker\beta=\Ker\alpha$. Then 
\[
\ell(\beta(U)) = \ell(U)-\ell(\Ker\beta) = \ell(U)-\ell(\Ker\alpha) 
\]
\[
= \ell(U) - \ell(\Bild\phi) + \ell(\Bild\phi')  = \ell(\coKer\phi) + \ell(U') - \ell(\coKer\phi')
\le \ell(U')
\]
by the inequality. Thus $\beta(U)\subseteq U'$, and one can take $\theta$ to be the restriction of $\beta$ to $U$.

(iii) Follows from (i) and (ii).
\end{proof}

In view of the lemma, when $U$ is indecomposable, a non-zero morphism $\phi:X\to U$ is determined (up to an isomorphism of $U$) by the pair of natural numbers $(s,t) = (\ell(\Ker\phi),\ell(\coKer\phi))$. We denote a representative of this morphism by $\phi_{st}:X\to X(s,t)$. Clearly if $s\le s'<\ell(X)$, then there is map $p_{st}^{s'}:X(s,t)\to X(s',t)$, necessarily an epimorphism, with $\phi_{s',t} = p_{st}^{s'} \phi_{st}$. 

Now let $M$ be an arbitrary module. We define $M_X$ to be the set of pairs $(s,t)$ such that $X(s,t)$ is a direct summand of $M$. The set $M_X$ inherits the partial ordering from $\Z^2$.

\begin{lem}
\label{l:Nakminapprox}
The map 
\[
\phi = (\phi_i) : X \to \bigoplus_{i=1}^k X(s_i,t_i),
\]
is a minimal left $\add(M)$-approximation of $X$, where  $(s_1,t_1),\dots,(s_k,t_k)$ are the minimal elements of $M_X$, ordered so that $s_1 > \dots > s_k$ and $t_1 < \dots < t_k$, and $\phi_i = \phi_{s_i,t_i}$. Assuming that $k>0$, or equivalently that $\Hom(X,M)\neq 0$, we have
\[
\coKer\phi \cong C_1\oplus \dots \oplus C_k
\]
where $C_1$ is the quotient of $X(s_1,t_1)$ of length $t_1$ and $C_i = X(s_{i-1},t_i)$ for $i>1$.
\end{lem}

\begin{proof}
The fact that $\phi$ is a left approximation follows immediately from part (ii) of Lemma \ref{l:3.1}. To show that $\phi$ is a minimal approximation, it suffices to show that if $\theta:X\to M'$ is a minimal $\add(M)$-approximation of $X$, then each $X(s_i,t_i)$ is a summand of $M'$. Now up to isomorphism we may write $M'$ as a direct sum of modules $X(s,t)$ for various $(s,t)$, with the components of $\theta$ being the maps $\phi_{st}$. By assumption the map $\phi_i$ factors through $\theta$. Consider a composition 
\[
X \xrightarrow{\phi_{st}} X(s,t) \xrightarrow{\alpha} X(s_i,t_i).
\]
If $s>s_i$ then the first map has kernel of length $>s_i$, and hence so does the composition, so $\ell(\Bild \alpha\phi_{st}) < \ell(X)-s_i$. If $t>t_i$ then $\alpha$ has kernel of length at least $\ell(X(s,t))-\ell(X(s_i,t_i)) = -s+t+s_i-t_i$, and so
\[
\ell(\alpha(\Bild \phi_{st})) \le \max\{\ell(X)+ t_i - t - s_i,0\} < \ell(X) - s_i.
\]
Since the map $\phi_i$ factors through $\theta$, and it has image of length $\ell(X)-s_i$, we deduce that some summand $X(s,t)$ has $(s,t)\le (s_i,t_i)$. By minimality $(s,t)=(s_i,t_i)$, so $X(s_i,t_i)$ must occur as a summand of $M'$. Since the modules $X(s_i,t_i)$ have distinct lengths, we deduce that they are all summands of $M'$, as required.

Now suppose $k>0$. Let $\pi_1:X(s_1,t_1)\to C_1$ be the projection. For $i>1$, let $\pi_i=p^{s_{i-1}}_{s_i,t_i}: X(s_i,t_i)\to X(s_{i-1},t_i)$, then the composition $\pi_i \phi_i$ is non-zero, it has kernel of length $s_{i-1}$ and cokernel of length $t_i$, and $(s_{i-1},t_{i-1})\le (s_{i-1},t_i)$, so by Lemma \ref{l:3.1} (ii) there is map $\sigma_i:X(s_{i-1},t_{i-1})\to C_i$ with $\pi_i\phi_i = \sigma_i \phi_{i-1}$. This gives a sequence
\[
X \xrightarrow{\phi} \bigoplus_{i=1}^k X(s_i,t_i) \xrightarrow{\psi} \bigoplus_{i=1}^k C_i \to 0
\]
where
\[
\psi = \begin{pmatrix}
\pi_1 & 0 & \\ 
-\sigma_2 & \pi_2 & \\
& & & \ddots \\
& & & & \pi_{k-1} & 0 \\
0 & & & & -\sigma_k & \pi_k
\end{pmatrix}.
\]
Since the $\pi_i$ are all epimorphisms, so is $\psi$. Also, it's easy to check that $\Bild \phi=\Ker\psi$ and hence the sequence is exact.
\end{proof}

We may use Lemma~\ref{l:Nakminapprox} to compute $\cogen^1(M)$ for a module $M$, written as a direct sum of indecomposable modules, say $M = M_1\oplus \dots \oplus M_m$. Let $X$ be an indecomposable module and let $\coKer \phi = C_1\oplus\dots\oplus C_k$ as in Lemma~\ref{l:Nakminapprox}.

\begin{lem}
\label{l:cogenchar}
We have the following for an indecomposable module $X$.
\begin{itemize}
\item[(i)]$X\in\cogen(M)$ $\Leftrightarrow$ $X$ is isomorphic to a submodule of $M_j$ for some $j$.
\item[(ii)]$X\in\cogen^1(M)$ $\Leftrightarrow$ $X,C_1,\dots,C_k$ are in $\cogen(M)$.
\end{itemize}
\end{lem}

\section{(Faithfully) balanced modules for Nakayama algebras}
\label{s:fbcond}

In this section $\Lambda$ is again a Nakayama algebra. Recall that a module $X$ is a \emph{subquotient} of $Y$ if $X \cong Y''/Y'$ for some submodules $Y' \subseteq Y'' \subseteq Y$; it is a \emph{proper} subquotient if $Y'\neq 0$ or $Y'' \neq Y$.

\begin{thm}
\label{t:bal_nakayama}
If $\Lambda$ is Nakayama, then a module $M$ is balanced if and only if it satisfies the following two conditions:
\begin{itemize}
\item[(B1)]if $X$ is an indecomposable summand of $M$ and $X$ is a proper subquotient of some indecomposable summand of $M$, then $X$ is a proper submodule or proper quotient of some indecomposable summand of~$M$, and
\item[(B2)]if $S,T$ are simple modules with $\Ext^1(T,S)\neq 0$ and $S$ or $T$ is a composition factor of $M$, then $\Hom(M,S)\neq 0$ or $\Hom(T,M)\neq 0$. 
\end{itemize}
\end{thm}

\begin{proof}
Assuming that $M$ is balanced, we prove (B1). Let $U$ be an indecomposable direct summand of $M$ which is a proper subquotient of some other indecomposable summand, and suppose that $U$ is not a proper submodule or quotient of any indecomposable summand of $M$. We derive a contradiction. Let $\theta:P\to U$ be the projective cover of $U$. Clearly $P$ is indecomposable and we consider the minimal left $\add(M)$-approximation $\phi$ of $P$ given by Lemma~\ref{l:Nakminapprox}, involving modules of the form $P(s,t)$. Letting $s = \ell(\Ker\theta)$, we can identify $U = P(s,0)$ and $\theta = \phi_{s0}$. Then $(s,0)\in M_P$, and it is minimal since if $s'<s$ then $U$ is a proper quotient of $P(s',0)$. Now $U$ is a proper subquotient of some indecomposable summand $Y$ of $M$. Say $U \cong Y''/Y'$ for $Y' \subset Y'' \subset Y$. Then $\theta$ lifts to a map $P\to Y''$, and by inclusion this gives a map $\psi:P\to Y$. We write it in the form $\phi_{s',t'}:P\to P(s',t')$. Now $(s,0)$ is incomparable with $(s',t')$, so the minimal approximation of $P$ has at least $k\ge 2$ terms, and the first term must be $(s_1,t_1) = (s,0)$. The second term $(s_2,t_2)$ gives rise to a summand $C_2 = P(s,t_2)$ of $\coKer\phi$. Now $U = P(s,0)$ embeds via the map $i_{s,t_2}^0$ in $C_2$. By assumption $M$ is balanced, so $\coKer\phi$ is cogenerated by $M$. Thus $C_2$ and hence also $U$ embeds in an indecomposable summand of $M$, a contradiction. Thus there can be no such summand $U$, so (B1) holds.

Next, assuming still that $M$ is balanced, we prove (B2). First assume that $S$ is a composition factor of $M$. Let $\phi$ be the minimal $\add(M)$-approximation of the projective cover $P$ of $S$ given by Lemma~\ref{l:Nakminapprox}. Since $\Hom(P,M)\neq 0$ we have $k>0$ in the lemma. Now if $\Hom(M,S)=0$, then $\Hom(P(s_1,t_1),S)=0$, so $S$ is not the top of $P(s_1,t_1)$, so $\phi_1$ is not surjective. Thus $C_1 = \coKer\phi_1$ is non-zero with socle $T$. Since $M$ is balanced, $C_1$ embeds in $M$, hence $\Hom(T,M)\neq 0$. On the other hand, if $T$ is a composition factor of $M$, then since the dual module $DM$ is a balanced $\Lambda^{op}$-module and $\Ext^1(DS,DT)\neq 0$, this argument shows that $\Hom(DM,DT)\neq 0$ or $\Hom(DS,DM)\neq 0$, so $\Hom(T,M)\neq 0$ or $\Hom(M,S)\neq 0$, as required.

For the converse, we now assume that (B1) and (B2) hold. Fix a simple $S$ and consider the minimal left $\add(M)$-approximation $\phi$ of the projective cover $P$ of $S$ as in Lemma~\ref{l:Nakminapprox}. We need to show that the summands $C_i$ of $\coKer\phi$ are in $\cogen(M)$. If $k=0$ there is nothing to check, so suppose that $k>0$. Thus $S$ is a composition factor of $M$.

First we consider the term $C_1$. If $\Hom(M,S)\neq 0$, then $M$ has a summand with top $S$. It follows that the first of the minimal elements of $M_P$ is of the form $(s,0)$. But then $C_1=0$, so there is nothing to check for this term. On the other hand, if $\Hom(M,S)=0$, then the first of the minimal elements of $M_P$ is of the form $(s_1,t_1)$ with $t_1\neq 0$. Then $C_1$ is non-zero, say with socle $T$. Clearly $\Ext^1(T,S)\neq 0$, so by condition $(B2)$ there is an indecomposable summand $U$ of $M$ with socle $T$, and, say, length $h$. Take $h$ maximal with this property. If $h\ge \ell(C_1)$, then $C_1$ embeds in $U$, as required. Otherwise $h < \ell(C_1)$. Then $U$ embeds in $C_1$, so it is a proper subquotient of $P(s_1,t_1)$. Thus by (B1), $U$ is a proper quotient or submodule of a summand $U'$ of $M$. Both are impossible. Indeed, if there is a proper surjection $\alpha:U'\to U$, then the top of $\Ker\alpha$ is $S$, so $\Ker\alpha$ is the image of a map $\psi:P\to U'$. But then $\ell(\coKer\psi) = \ell(U) = h < \ell(C_1) = t_1 \le t_r$ for all $r$. This is impossible since $U' \cong P(s,h)$ for some $s$ so $(s,h)\in M_P$, contradicting the fact that the $(s_i,t_i)$ are the minimal elements. If $U$ is a proper submodule of $U'$ then $h$ was not maximal.

Next we consider the term $C_i$ for $1<i\le k$. It is a quotient of $P(s_i,t_i)$ and it has a submodule isomorphic to $P(s_{i-1},t_{i-1})$. Thus $P(s_{i-1},t_{i-1})$ is a proper subquotient of an indecomposable summand of $M$. Since $(s_{i-1},t_{i-1})$ is a minimal element of $M_P$, it follows that $P(s_{i-1},t_{i-1})$ is not a proper quotient of any indecomposable summand of $M$. Thus by (B1) it is a proper submodule of an indecomposable summand $U'$ of $M$. Take $\ell(U')$ to be maximal. If $\ell(U')\ge \ell(C_i)$, then $C_i$ embeds in $U'$, as required. Thus for a contradiction suppose that $\ell(U')<\ell(C_i)$. Then $U'$ properly embeds in $C_i$. Thus $U'$ is a subquotient of $P(s_i,t_i)$, so by condition (B1), $U'$ is a proper submodule or quotient of an indecomposable summand $W$ of $M$. If it is a proper submodule of $W$, then $\ell(U')$ is not maximal. Thus $U'$ is a proper quotient of $W$. Now the composition $f$ of $\phi_{i-1}$ with the inclusion $P(s_{i-1},t_{i-1})\to U'$ lifts to a map $g:P\to W$. Then $\ell(\Bild g) >\ell(\Bild f) = \ell(\Bild \phi_{i-1})$ and $\ell(\coKer g) = \ell(\coKer f) > \ell(\coKer \phi_{i-1})$. By assumption $(\ell(\Ker g),\ell(\coKer g)) \ge (s_j,t_j)$ for some $j$. We must have $j>i$ since $\ell(\Bild g) > \ell(\Bild \phi_{i-1})$. On the other hand, $\ell(\coKer g) = \ell(\coKer f) < \ell(\coKer \alpha) = \ell(\coKer \phi_i)$, where $\alpha$ is the composition of $\phi_{i-1}$ with the inclusion $U'\to C_i$. Thus $j\ge i$ is not possible.
\end{proof}

We recall that Nakayama algebras are QF-3, so a module is faithful if and only if it has every indecomposable projective-injective module as a summand, see \cite[Theorem 32.2]{AF74}. For faithfully balanced modules, Theorem~\ref{t:bal_nakayama} takes the following form.

\begin{cor}
\label{c:fb_nakayama}
If $\Lambda$ is Nakayama, then a module $M$ is faithfully balanced if and only if it satisfies the following conditions:
\begin{itemize}
\item[(FB0)]every indecomposable projective-injective module is a summand of $M$,
\item[(FB1)]if $X$ is an indecomposable summand of $M$ and $X$ is not projective-injective, then $X$ is a proper submodule or proper quotient of some indecomposable summand of~$M$, and
\item[(FB2)]if $S,T$ are simple modules with $\Ext^1(T,S)\neq 0$, then $\Hom(M,S)\neq 0$ or $\Hom(T,M)\neq 0$. 
\end{itemize}
\end{cor}

Specializing to the algebra $\Lambda_n$, which is a Nakayama algebra, this gives Theorem~\ref{t:fban}.

\section{Critical modules and Minimal faithfully balanced modules}
\label{s:crit}
Let $\Lambda$ be an artin algebra. 

\begin{lem}
\label{l:masauterprop}
Given modules $N,M$, we have 
\begin{itemize}
\item[(i)]$N\in \gen_1(M)$ if and only if $\gen_1(M\oplus N) = \gen_1(M)$.
\item[(ii)]$N\in \cogen^1(M)$ if and only if $\cogen^1(M\oplus N) = \cogen^1(M)$.
\end{itemize}
\end{lem}

\begin{proof}
Part (ii) is due to Ma and Sauter \cite[Lemma 3.3]{MS19}, and part (i) is dual.
\end{proof}

Recall that a (faithfully balanced) \emph{$\gen_1$-category} is a subcategory of $\lmod{\Lambda}$ of the form $\gen_1(M)$, where $M$ is a (faithfully balanced) module.

\begin{proof}[Proof of Theorem \ref{t:gencrit}]
Clearly $\mcG$ contains at least one $\gen_1$-critical module $M$ with $\gen_1(M)=\mcG$. We shall show that $M$ is uniquely determined.

By assumption $\Lambda$ is representation-directed \cite[\S IX.3]{ASS}, so we can enumerate the indecomposable modules in $\mcG$ as $X_1,X_2,\dots,X_m$ with $\Hom(X_j,X_i)=0$ for $j>i$ and each $\End(X_i)$ a division algebra.

We show by induction on $i$ how to determine whether or not $X_i$ is a summand of $M$. Let $M_i$ be the direct sum of all $X_j$ with $j<i$ which occur as summands of $M$. By the inductive hypothesis this is uniquely determined.

We show that $X_i$ is a summand of $M$ if and only if $X_i \notin \gen_1(M_i)$. Namely, if $X_i$ is a summand of $M$, write $M = M'\oplus X_i$. 
By the ordering of the $X_i$, the minimal right $\add(M')$-approximation of $X$ is the same as the minimal right $\add(M_i)$-approximation. Thus if $X_i \in\gen_1(M_i)$, then $X_i\in\gen_1(M')$. But then $\gen_1(M')=\gen_1(M'\oplus X_i) = \gen_1(M)$ by Lemma \ref{l:masauterprop}, contradicting the criticality of $M$. Conversely, if $X_i$ is not a summand of $M$, then the minimal right $\add(M)$-approximation of $X$ is the same as the minimal right $\add(M_i)$-approximation, so if $X_i \notin \gen_1(M_i)$ then $X_i\notin\gen_1(M) = \mcG$, which is nonsense. 

The final part of the theorem follows from Lemma~\ref{l:masauterprop}.
\end{proof}

In our example in the introduction, we have illustrated the faithfully balanced $\gen_1$-categories for~$\Lambda_n$ with $n=3$. The next proposition shows that in order to understand arbitrary $\gen_1$-categories for~$\Lambda_n$ it is equivalent to understand faithfully balanced $\gen_1$-categories for $\Lambda_{n+1}$. Let $\mathcal{C}$ be the category of $\Lambda_{n+1}$-modules vanishing at vertex $1$, and  $F:\mathcal{C}\to\lmod{\Lambda_n}$ the equivalence of categories which forgets vertex $1$.

\begin{pro}
The assignment $\mathcal{G}\mapsto F(\mathcal{G}\cap \mathcal{C})$ gives a 1:1 correspondence between faithfully balanced $\gen_1$-categories for $\lmod{\Lambda_{n+1}}$ and arbitrary $\gen_1$-categories for $\Lambda_n$. The inverse  sends $\mathcal{H}$ to the category of $\Lambda_{n+1}$-modules which are the direct sum of an injective module and a module in $\mathcal{C}$ whose image under $F$ is in $\mathcal{H}$.
\end{pro}

\begin{proof}
Observe that the indecomposable modules for $\Lambda_{n+1}$ are either injective or in $\mathcal{C}$. If $M \in \mathcal{C}$ and $I$ is injective, it is easy to see that a module $X\in \mathcal{C}$ is in $\gen_1(M\oplus I)$ if and only if $F(X)\in\gen_1(F(M))$. Using Lemma~\ref{l:masauterprop}, the result follows.
\end{proof}

According to our computer calculations, the number of faithfully balanced $\gen_1$-categories in $\lmod{\Lambda_n}$ for $n=1,\dots,6$ is $1,2,7,39,325,3875$, and the number of minimal faithfully balanced $\Lambda_n$-modules is $1,2,6,25,134,881$.

\begin{lem}
\label{l:rigidpdone}
A basic module $T$ with $\pd T\leq 1$ and which is rigid (i.e.\ $\Ext^1(T,T)=0$), is $\gen_1$-critical. 
\end{lem}

\begin{proof}
Assume $T=M\oplus N$ and $\gen_1(M)=\gen_1(T)$. Then we have $N\in \gen_1(M)$ and so there is an exact sequence $M_1\to M_0\to N\to 0$ with $M_0, M_1\in \add(M)$ and $\Hom(M,-)$ exact on it. Thus we obtain two short exact sequences
\[
\begin{aligned}
0 \to X_1 \to &M_1 \to X_0 \to 0, \\
0 \to X_0 \to &M_0 \to N \to 0.
\end{aligned}
\]
Applying $\Hom(N, -)$ to the first exact sequence yields an exact sequence
\[
0=\Ext^1(N, M_1)\to \Ext^1(N, X_0)\to \Ext^2(N, X_1)=0
\]
since $T$ is rigid and $\pd N\leq \pd T\leq 1$. This means the second short exact sequence is split and so $N\in \add(M)$. It follows that $\add(M)=\add(T)$ and therefore $M=T$ since $T$ is basic. 
\end{proof}

\begin{proof}[Proof of Theorem \ref{t:tilting}]
The first part is a special case of Lemma \ref{l:rigidpdone}.
Now suppose that $T$ is a basic tilting module and $\Lambda$ is hereditary. Let $M$ be a faithfully balanced summand of $T$. Then we have two exact sequences 
\[
\begin{aligned}
0 \to \La \to &M_0 \to X \to 0, \\
0 \to X \to &M_1 \to Y \to 0
\end{aligned}
\]
with $M_i\in \add (M)$ such that $\Hom_{\La} (-, M)$ is exact on both short exact sequences. It is straightforward to check that $T'=M \oplus X$ is a tilting module. By definition $T' \in \gen(T) \cap \cogen (T)=T^{\perp} \cap {}^{\perp} T$, so $T \oplus T'$ is rigid and since tilting modules are maximal rigid we conclude that $\add (T)= \add (T')$. By applying $\Hom (-,M)$ to the second short exact sequence we obtain $\Ext^1(Y,M)=0$. By applying $\Hom (Y,-)$ to the first exact sequence we obtain $\Ext^1(Y,X)=0$. Thus the second short exact sequence splits, so $X \in \add(M)$. This implies $\add (T) = \add(M \oplus X) = \add (M)$, and since $T$ is basic, it follows that $M=T$.
\end{proof}

We refer to \cite{AIR} for the notion of a support $\tau$-tilting module.

\begin{cor}
Every basic support $\tau$-tilting module is $\gen_1$-critical and balanced. 
\end{cor}

\begin{proof}
Let $I=\rm{ann} (M)$ be the annihilator ideal of a support $\tau$-tilting module $M$. By \cite[Proposition 2.2]{AIR}, ${}_{\La/I}M$ is a classical tilting module, so it is faithfully balanced and $\gen_1$-critical as a $\La/I$-module by Theorem~\ref{t:tilting}. This implies that ${}_{\La}M$ is balanced. Assume $N \in \add({}_{\La}M)$ such that $\gen_1(M)=\gen_1(N)$ and consider the fully faithful and exact functor $i\colon \La/I\!-\!\modu \to \La \!-\!\modu$ which has a left adjoint $q= \La / I \otimes_{\La} - $. Now we have $M \cong iq(M)$, therefore $N =i(N')\in \add(M) \subset \Bild i$. 
We have $i(\gen_1 (N'))= \gen_1(N)\cap \Bild i = \gen_1(M) \cap \Bild i = i (\gen_1({}_{\La/I}M))$ and since $i$ is fully faithful, $\gen_1 (N')=\gen_1({}_{\La/I}M)$. This implies $\add (N')=\add ({}_{\La/I} M)$ since ${}_{\La/I}M$ is $\gen_1$-critical and then apply $i$ to conclude $\add(N)= \add (M)$. This proves $M$ is $\gen_1$-critical. 
\end{proof}

The following result is due to Morita \cite[Theorem 1.1]{Morita}.

\begin{thm}
\label{t:morita}
If $M$ is a faithfully balanced module for an algebra $\Lambda$ and $X$ is indecomposable, then $M\oplus X$ is faithfully balanced if and only if $X\in\gen(M)$ or $X\in\cogen(M)$.
\end{thm}

For convenience we give the proof of this in the next two lemmas. Observe that  
one direction, Lemma~\ref{Mor1}, holds with the weaker assumption that $M$ is faithful.

\begin{lem}\label{Mor1}
Let $M$ be faithful and $X$ be indecomposable. If $M\oplus X$ is (faithfully) balanced, then we have either $X\in \gen(M)$ or $X\in\cogen(M)$.
\end{lem}
\begin{proof}
Let $E=\End_{\La}(X)$, then $E$ is a local ring and hence there is a unique simple $E$-module, say $S$. Define $X_1=\sum_{f:M\to X} \im(f)$ and $X_0=\bigcap_{g:X\to M} \Ker(g)$. Then $X_1$ and $X_0$ are submodules of $X$. By definition, we have $X\in \gen(M)$ if and only if $X_1=X$ and $X\in \cogen(M)$ if and only if $X_0=0$.
Now assume $X_1\neq X$ and $X_0\neq 0$. Then $X/X_1\neq 0$ and hence has $S$ as a quotient. This implies $\Hom_E(X/X_1, X_0)\neq 0$. Thus there exists a non-zero $E$-endomorphism $\theta: X\to X$ such that $X_1\subseteq \Ker(\theta)$ and $\im(\theta)\subseteq X_0$. Let $\Ga=\End_{\La}(M)$, then we have 
$$\End_{\La}(M\oplus X)=\begin{pmatrix} \Ga & \Hom_{\La}(X, M)\\ \Hom_{\La}(M, X) & E
\end{pmatrix},$$
and $M\oplus X$ is a left $\End_{\La}(M\oplus X)$-module. We claim that $\begin{pmatrix} 0 & 0 \\ 0 & \theta \end{pmatrix} $ is an $\End_{\La}(M\oplus X)$-endomorphism of $M\oplus X$, that is, for any element $\begin{pmatrix} a & b \\ c & d \end{pmatrix} \in \End_{\La}(M\oplus X)$ we have $$\begin{pmatrix} 0 & 0 \\ 0 & \theta \end{pmatrix} \begin{pmatrix} a & b \\ c & d \end{pmatrix} =\begin{pmatrix} a & b \\ c & d \end{pmatrix} \begin{pmatrix} 0 & 0 \\ 0 & \theta \end{pmatrix}.$$ 
To prove the claim we need to show $\theta c=0$, $b\theta=0$ and $\theta d=d\theta$. Now $\im(c)\subseteq X_1\subseteq \Ker(\theta)$ gives $\theta c=0$, $\im(\theta)\subseteq X_0$ gives $b\theta=0$ and the fact that $\theta$ is an $E$-endomorphism gives $\theta d=d\theta$. By assumption, $M\oplus X$ is balanced and this implies that the action of $\begin{pmatrix} 0 & 0 \\ 0 & \theta \end{pmatrix} $ is given by the multiplication of some element $\lambda\in \La$. Now we must have $\lambda M=0$ which forces $\lambda=0$ since $M$ is faithful as a $\La$-module. Thus we have $\theta=0$, a contradiction.
\end{proof}

\begin{lem}\label{Mor2}
Let $M$ be faithfully balanced. If either $X\in \gen(M)$ or $X\in\cogen(M)$, then $M\oplus X$ is also faithfully balanced.
\end{lem}
\begin{proof}
We will prove the case $X\in \gen(M)$; the case $X\in \cogen(M)$ is dual. Since $M$ is faithfully balanced, there is an exact sequence
$$0\to \La \xrightarrow{f} M_0 \xrightarrow{g} M_1$$
such that $f$ and $\coKer(f)\to M_1$ are minimal left $\add(M)$-approximations. We claim that the map $f$ is also a left $\add(M\oplus X)$-approximation. To this end, it is enough to show that any map $h:\La\to X$ factors through $f$. Consider the following diagram
$$\xymatrix{0\ar[r]& \La \ar@{-->}[ld]_{i}\ar[d]_(0.4){h}\ar[r]^{f} & M_0\ar@{-->}[lld]^(0.4){j} \\  M^X \ar[r]^{p} & X \ar[r] & 0}$$
where $p$ is the minimal right $\add(M)$-approximation of $X$. Since $X\in \gen(M)$, $p$ is an epimorphism and so there is an $i:\La\to M^X$ such that $h=pi$. Then $i$ factors as $i=jf$ and we have $h=pi=(pj)f$. This proves the claim. Now since $\coKer(f)\in \cogen(M)\subseteq \cogen(M\oplus X)$ we conclude that $\La \in \cogen^1(M\oplus X)$. This proves $M\oplus X$ is faithfully balanced.
\end{proof}

For Nakayama algebras, the conditions (FB1) and (FB2) in Corollary~\ref{c:fb_nakayama} allow a different approach to minimal faithfully balanced modules. We begin with some constructions which work for a module $M$ for an arbitrary algebra. Recall \cite{ASm1} that a module $X\in\add(M)$ is a \emph{splitting projective} if every epimorphism $M'\to X$ with $M'\in\add(M)$ is a split epimorphism, and it is a \emph{splitting injective} if every monomorphism $X\to M'$ is a split monomorphism. We write $M^g$ for the direct sum of one copy of each of the splitting projective summands of $M$ and $M^c$ for the direct sum of one copy of each of the splitting injective summands of $M$. By \cite[Theorem 2.3]{ASm1}, $\add(M^g)$ is a minimal cover for $\add(M)$, so $M^g$ is a minimal summand of $M$ with $\gen(M^g) = \gen(M)$, and it is unique up to isomorphism with this property. Similarly for $M^c$ with $\cogen(M^c)=\cogen(M)$.

For Nakayama algebras, Morita's Theorem~\ref{t:morita} can be used to construct all faithfully balanced modules from minimal faithfully balanced modules. This follows from the following lemma:

\begin{lem} 
Let $\La$ be a Nakayama algebra. If $M$ is faithfully balanced but not minimal faithfully balanced, then there is a faithfully balanced summand $N$ of $M$ with $\lvert M \rvert =\lvert N \rvert +1$. 
\end{lem}

\begin{proof}
Let $L$ be a faithfully balanced proper summand of $M$. Let $M=L\oplus U$. Pick an indecomposable summand $U'\in \add (U)$ of minimal length and let $U=U' \oplus V$. 
Then $N:=V\oplus L$ still fulfills the cohook conditions and therefore is a faithfully balanced module. Indeed, the condition (FB2) is satisfied by the summand $L$ and the hypothesis on the length of $U'$ implies that no other indecomposable modules are generated or co-generated by $U'$ so condition (FB1) also holds. 
\end{proof}

\begin{rem}
We don't know whether this result holds without the assumption that $\Lambda$ is Nakayama.
\end{rem}

\begin{lem}
\label{l:nakgc}
\label{l:cc+gc}
If $M$ is a minimal faithfully balanced module for a Nakayama algebra $\Lambda$, then
any indecomposable summand $X$ of $M$ is a summand of $M^g$ or $M^c$, and $X$ is a summand of both if and only if $X$ is projective-injective. Thus
\[
M \oplus P \cong M^g \oplus M^c,
\]
where $P$ is the direct sum of the indecomposable projective-injective $\Lambda$-modules.
\end{lem}

\begin{proof}
Since $M$ is faithfully balanced, by condition (FB1) in Corollary~\ref{c:fb_nakayama}, every indecomposable summand $X$ of $M$ which is not projective-injective is a proper submodule or quotient of another summand of $M$. Thus $X$ cannot be a summand of both $M^g$ and $M^c$. On the other hand, if $X$ is a summand of neither, then it is both a proper submodule and quotient of other summands of $M$. But then the complement of $X$ still satisfies the conditions of Corollary~\ref{c:fb_nakayama}, so is faithfully balanced, contradicting minimality.
\end{proof}

\begin{thm}
\label{t:mfb-gen-cogen}
Let $M$ be a minimal faithfully balanced module for a Nakayama algebra $\Lambda$. If $N$ is a module with
$\gen(N)\cap\cogen(N) = \gen(M)\cap\cogen(M)$, then $N$ is faithfully balanced and $M$ is a summand of $N$.
\end{thm}

\begin{proof}
Clearly $\gen(\gen(M)\cap\cogen(M)) = \gen(M)$ and $\cogen(\gen(M)\cap\cogen(M)) = \cogen(M)$, so we have $\gen(N)=\gen(M)$ and $\cogen(N)=\cogen(M)$. By the uniqueness of minimal covers and cocovers, $M^g\cong N^g$ and $N^c \cong M^c$. By Lemma~\ref{l:nakgc}, we conclude that $M$ is a summand of $N$. Now $N$ is faithfully balanced by Theorem~\ref{t:morita}.
\end{proof}

\section{Counting faithfully balanced modules}
\label{s:fbcount}
In this section we prove Theorem~\ref{t:fbcount}. Given a module $M$ for $\Lambda_n$, we write $t_r(M)$ for the number non-isomorphic indecomposable summands of $M$ with top $S[r]$, or equivalently in row $r$ in the Young diagram. We consider indeterminates $x_1,\dots,x_n$, and define
\[
k_n(x_1,\dots,x_n) 
= \sum_M \prod_{r=1}^n x_r^{t_r(M)} \in \Z[x_1,\dots,x_n]
\]
where the sum is over all basic faithfully balanced $\Lambda_n$-modules $M$. We define
\[
p_n(x_1,\dots,x_n) 
= \sum_M \prod_{r=1}^n x_r^{t_r(M)} \in \Z[x_1,\dots,x_n]
\]
where the sum is over all modules $M$ satisfying (FB0) and (FB1) in the statement of Theorem~\ref{t:fban}.

Let $[2,n] := \{ k\in\Z : 2\le k\le n\}$. In condition (FB2) in Theorem~\ref{t:fban}, it follows from (FB0) that the $M$ contains summands in the virtual cohooks associated to the leaves $(1,0)$ and $(n+1,n)$. Thus we may replace (FB2) by the conditions (FB2)${}_k$ that $M$ has a summand in $\cohook(k,k-1)$, for all $k\in [2,n]$. Given a subset $I\subseteq [2,n]$, we define $s_n^I(x_1,\dots,x_n)$ to be the sum of $\prod_{r=1}^n x_r^{t_r(M)}$ over all $M$ which satisfy (FB0), (FB1) and (FB2)${}_k$ for all $k\in I$, and $f_n^I(x_1,\dots,x_n)$ to be the sum over all $M$ which satisfy (FB0), (FB1) and fail (FB2)${}_k$ for all $k\in I$.

We write $\underline{x}_n = (x_1,\dots,x_n)$ and for a subset $J = \{j_1 <\dots < j_m \}$ of $[2,n]$, we write 
\[
\underline{x}_n^J = (x_1,\dots,\hat x_{j_1},\dots,\hat x_{j_m},\dots,x_n)
\]
where $\hat x_p$ means that the term $x_p$ is omitted.

\begin{lem}\label{l:6.1}
We have the following.
\begin{itemize}
\item[(i)]$f_n^I  (\underline{x}_n) = p_{n-|I|} (\underline{x}_n^I)$.
\item[(ii)]$s_n^I (\underline{x}_n) 
= \sum_{J\subseteq I} (-1)^{|J|} p_{n-|J|} (\underline{x}_n^J)$.
\item[(iii)]$k_n(\underline x_n)
= \sum_{J\subseteq [2,n]} (-1)^{|J|} p_{n-|J|} (\underline{x}_n^J)$.
\item[(iv)]$p_n (\underline{x}_n) 
= \sum_{J\subseteq [2,n]} k_{n-|J|}(\underline{x}_n^J)$.
\end{itemize}
\end{lem}

\begin{proof}
(i) To fail the condition (FB2)${}_k$ means that row $k$ and column $k-1$ of the Young diagram must be empty. If so we can shrink the diagram to obtain a Young diagram for a smaller $n$. 

(ii) Follows by the inclusion-exclusion principle.

(iii) This is a special case of (ii).

(iv) This follows by another application of the inclusion-exclusion principle.
\end{proof}

\begin{lem}\label{l:6.2}
We have
\[
p_{n+1}(\underline x_{n+1}) = 
\left( \prod_{i=1}^{n+1} (1+x_i) \right)
 \sum_{I \subseteq [2,n]} k_{n-|I|} (\underline x_n^I) \cdot \prod_{i\in I} \frac{1}{1+x_i}.
\]
\end{lem}

\begin{proof}
Let $I \subseteq [1,n+1]$. Given a basic module $M$ for $\Lambda_n$, we obtain a module for $\Lambda_{n+1}$ of the form
\[
M' = 
\left( \bigoplus_{i\in I} M_{ii} \right)
\oplus
\left( \bigoplus_{M_{ij}\in\add(M)} M_{i,j+1}\right).
\]
Moreover $M'$ satisfies (FB0) and (FB1) if and only if $M$ satisfies (FB0), (FB1) and (FB2)${}_k$ for $k\in [2,n]\cap I$.
Thus
\[
p_{n+1}(\underline x_{n+1}) 
= (1+x_1)(1+x_{n+1})\cdot \sum_{I\subseteq [2,n]} \left( \prod_{i\in I} x_i \right) s_n^I (\underline x_n) .
\]
By Lemma \ref{l:6.1}(ii) this becomes
\[
(1+x_1)(1+x_{n+1})\cdot \sum_{I\subseteq [2,n]} 
 \left( \prod_{i\in I} x_i \right)  
\left( \sum_{J\subseteq I} (-1)^{|J|} p_{n - |J|} (\underline x_n^J)\right) .
\]
Letting $L = I \setminus J$ we can rewrite this as
\[
(1+x_1)(1+x_{n+1})\cdot \sum_{J\subseteq [2,n]} 
 \left(\prod_{j\in J} x_j \right) (-1)^{|J|} p_{n - |J|} (\underline x_n^J) 
  \sum_{L\subseteq [2,n]\setminus J}  \left( \prod_{\ell\in L} x_\ell \right).
\]
\[
= (1+x_1)(1+x_{n+1})\cdot \sum_{J\subseteq [2,n]} 
 \left(\prod_{j\in J} x_j \right) 
 (-1)^{|J|} p_{n - |J|} (\underline x_n^J) \cdot
 \left( \prod_{i \in [2,n]\setminus J} (1+x_i) \right).
\]
\[
= \left( \prod_{i=1}^{n+1} (1+x_i) \right) \sum_{J\subseteq [2,n]} 
 (-1)^{|J|} \left(\prod_{j\in J} \frac{x_j}{1+x_j} \right) 
  p_{n - |J|} (\underline x_n^J).
\]
By part (iv) of Lemma \ref{l:6.1} this becomes
\[
\left( \prod_{i=1}^{n+1} (1+x_i) \right) \sum_{J\subseteq [2,n]} 
 (-1)^{|J|} \left(\prod_{j\in J} \frac{x_j}{1+x_j} \right) 
  \sum_{K \subseteq [2,n]\setminus J} k_{n - |J|-|K|} (\underline x_n^{J\cup K}).
\]
Letting $I = J\cup K$ this becomes
\[
\left( \prod_{i=1}^{n+1} (1+x_i) \right) 
\sum_{I\subseteq [2,n]} 
  k_{n - |I|} (\underline x_n^I)
 \sum_{J\subseteq I} 
 (-1)^{|J|} \left(\prod_{j\in J} \frac{x_j}{1+x_j} \right) 
\]
\[
= \left( \prod_{i=1}^{n+1} (1+x_i) \right) 
\sum_{I\subseteq [2,n]} 
  k_{n - |I|} (\underline x_n^I)
  \left( \prod_{i\in I} \frac{1}{1+x_i} \right)
\]
as claimed.
\end{proof}

\begin{proof}[Proof of Theorem \ref{t:fbcount}]
Define 
\[
h_n(\underline x_n) = 
\prod_{r=1}^n \left( \prod_{s=1}^r (1+x_s) - 1 \right)
\]
as in the statement of the theorem. Suppose by induction that $k_m(\underline x_m) = h_m(\underline x_m)$ for all $m\le n$. We show that $k_{n+1}(\underline x_{n+1}) = h_{n+1}(\underline x_{n+1})$. For a subset $I$ of $[2,n]$ we have
\[
\left( \prod_{i=1}^{n+1} (1+x_i) \right)
h_{n-|I|}(\underline x_n^I) \prod_{i\in I} \frac{1}{1+x_i}
=
h_{n-|I|}(\underline x_n^I) \prod_{i\in [1,n+1]\setminus I} (1+x_i)
\]
\[
= h_{n-|I|}(\underline x_n^I) \left( \prod_{i\in [1,n+1]\setminus I} (1+x_i) -1\right) 
+ h_{n-|I|}(\underline x_n^I) 
\]
\[ 
= h_{n-|I|+1}(\underline x_{n+1}^I) 
+ h_{n-|I|}(\underline x_{n+1}^{I \cup \{ n+1\}}).
\]
By Lemma \ref{l:6.2} and the inductive hypothesis this gives
\[
p_{n+1}(\underline x_{n+1}) = 
 \sum_{I \subseteq [2,n]} 
\left( h_{n-|I|+1}(\underline x_{n+1}^I) 
+ h_{n-|I|}(\underline x_{n+1}^{I \cup \{ n+1\}}) \right)
=
 \sum_{I \subseteq [2,n+1]} 
h_{n+1-|I|}(\underline x_{n+1}^I).
\]
On the other hand, by Lemma \ref{l:6.1}(iv),
\[
p_{n+1}(\underline x_{n+1})
=
 \sum_{I \subseteq [2,n+1]} 
k_{n+1-|I|}(\underline x_{n+1}^I).
\]
By the inductive hypothesis, we can equate terms, giving $k_{n+1}(\underline x_{n+1}) = h_{n+1}(\underline x_{n+1})$, as required.
\end{proof}

Recall the notation $[n]_q = 1+q+q^2+\dots+q^{n-1}$ and $[n]_q! = [1]_q [2]_q \dots [n]_q$.

\begin{cor}
\label{c:fbnumbers}
For $\Lambda_n$ we have the following.
\begin{itemize}
\item[(i)]If $k_{n,s}$ denotes the number of basic faithfully balanced modules with $s$ summands, then 
\[ \sum_s k_{n,s} x^s = \prod_{i=1}^n ((1+x)^i -1).\]
\item[(ii)]$k_{n,s} = \sum_{(j_1, j_2 , \dots , j_n) \colon 1 \leq j_r \leq r, \sum_{r=1}^n j_r=s} \binom{1}{j_1}\binom{2}{j_2} \dots \binom{n}{j_n}$.
\item[(iii)]The number of basic faithfully balanced modules is $[n]_2 !:=\prod_{i=1}^n (2^i -1)$. The number of faithfully balanced modules in which the indecomposable summands have multiplicity at most $m$ is $\prod_{i=1}^n ((1+m)^i -1)$.
\item[(iv)]Any basic faithfully balanced module for $\Lambda_n$ has at least $n$ indecomposable summands, and the number of basic faithfully balanced modules with $n$ indecomposable summands is $n!$.
\item[(v)]The direct sum of all indecomposable modules is a faithfully balanced module with $N=n(n+1)/2$ indecomposable summands; there are $N-1$ basic faithfully balanced modules with $N-1$ summands.
\end{itemize}
\end{cor}

\section{Number of faithfully balanced modules for quadratic Nakayama algebras}
\label{s:quadratic}
In order to study quadratic Nakayama algebras, we begin with a lemma about faithfully balanced modules for $\Lambda_n$.
For $n\ge 1$ and $0\le k\le 2$, we define $N_k(n)\in\N$ by
\[
N_0(n) = [n]_2!,
\quad
N_1(n) = 2^{n-1}[n-1]_2!,
\quad
N_2(n) = \begin{cases}
1 & (n\le 2) \\
2^{n-3} (2^n - 1) [n-2]_2! & (n \ge 3).
\end{cases}
\]
Recall that $S[n]$ is a simple projective module for $\Lambda_n$, and $S[1]$ is a simple injective.

\begin{lem}
\label{l:nknformula}
Fix a subset of the set $\{S[1],S[n]\}$ of cardinality $k$.
The number of basic faithfully balanced modules for $\Lambda_n$
having all of the modules in this subset as direct summands is $N_k(n)$.
\end{lem}

\begin{proof}
The case $k=0$ is Corollary \ref{c:fbnumbers}(iii). The two possible subsets of size $k=1$ give the same number of faithfully balanced modules by duality, so we may assume that the subset is $\{S[n]\}$. If the Young diagram for a faithfully balanced module has a non-empty second column (starting from the left), then the simple $S[n]$ is \emph{irrelevant} for the faithfully balanced condition in Theorem~\ref{t:fban}. On the other hand, if the second column is empty then $S[n]$ must be a direct summand of $M$. Let $M$ be a faithfully balanced module for $\Lambda_{n}$ with empty second column. Removing the second column and the simple $S[n]$ and shrinking the diagram gives a faithfully balanced module for $\Lambda_{n-1}$. In other words, there is a bijection between faithfully balanced modules for $\Lambda_{n-1}$ and faithfully balanced modules for $\Lambda_{n}$ with an empty second column. If we denote by $t_n$ the number of faithfully balanced modules for $\Lambda_n$ having non-empty second column and $S[n]$ as a summand, we have $N_0(n) = N_0(n-1) + 2t_n$, and the number of faithfully balanced modules with $S[n]$ as a direct summand is $N_0(n-1) + t_n = N_1(n)$. 

The case $k=2$ is similar, but slightly more technical so we only sketch the arguments. The case $n\le 2$ is clear, so assume $n\ge 3$. If we denote by $A$ the set of summands of $M$ in the second column and by $B$ the summands in the second row (starting from the top), we can split the set of faithfully balanced modules into $4$ subsets accordingly to the emptiness or non-emptiness of $A$ and $B$. 

Let $r$ be the number of faithfully balanced modules having $A\neq \emptyset$ and $B\neq\emptyset$ and having $S[1]$ and $S[n]$ as direct summands. In this case the modules $S[1]$ and $S[n]$ are both irrelevant for the condition of faithfully balanced module.  

Let $s$ be the number of faithfully balanced modules having $A\neq \emptyset$ and $B = \emptyset$ and having $S[1]$ and $S[n]$ as direct summands. In this case the module $S[n]$ is irrelevant for the condition of faithfully balanced module. 

Note that duality induces a bijection between the case $A\neq \emptyset, B = \emptyset$ and the case $A = \emptyset, B\neq \emptyset$. Moreover, the shrinking argument used in the first part shows that there is a bijection between the set of faithfully balanced modules having $A = \emptyset, B = \emptyset$ and the set of faithfully balanced modules for $\Lambda_{n-2}$. In other words, we have
\[
N_0(n) = 4 r + 4 s + N_0(n-2).
\]
Looking at the modules having $B = \emptyset$ and using a shrinking argument, we have
\[
N_0(n-1) = 2 s + N_0(n-2).
\]
Now the number of faithfully balanced modules with both $S[1]$ and $S[n]$ as summands is
\[
r + 2s + N_0(n-2) = N_2(n),
\]
as required.
\end{proof}

Now let $\Lambda$ be a \emph{quadratic Nakayama} algebra, say of the form $KQ/I$ where $Q$ is a linearly oriented quiver of type $A_n$ or $\tilde A_n$ (see \cite[Theorem 10.3]{SY11}), and $I$ is an admissible ideal generated by paths of length 2. Let $P_1,\dots,P_t$ be the indecomposable projective-injective $\Lambda$-modules, say of lengths $n_1,\dots,n_t$, let $\mathcal{G}= \gen(\kdual \Lambda)\cap \cogen(\Lambda)$ and let $g$ be the number of simple $\Lambda$-modules in $\mathcal{G}$, equivalently the number of simples wich occur as the socle of some $P_i$ and the top of some $P_j$. Define $k_i$ to be 2 if $\soc P_i$ and $\Top P_i$ are both in $\mathcal{G}$, 1 if only one is in $\mathcal{G}$, and 0 otherwise.

\begin{thm}
\label{t:quadratic}
If $\Lambda$ is a quadratic Nakayama algebra
as above, then the number of basic faithfully balanced $\Lambda$-modules is
$2^g N_{k_1}(n_1) \dots N_{k_t}(n_t)$.
\end{thm}

\begin{proof}
It is easy to see that the Auslander-Reiten quiver of $\Lambda$ is a concatenation of the Auslander-Reiten quivers of the algebras $\Lambda_{n_i}$. See Figure~\ref{ex_nakayama} for an example where $Q$ is of type $A_n$. In case $Q$ is of type $\tilde A_n$, the diagram is similar, but the bottom left and top right vertices must be identified, and the corresponding simple is in $\mathcal{G}$.

Corollary~\ref{c:fb_nakayama} implies that adding or deleting a simple in $\mathcal{G}$ as a summand of a module $M$ does not affect whether or not $M$ is faithfully balanced. Thus the number of basic faithfully balanced modules is $2^g$ times the number of those which have all simples in $\mathcal{G}$ as a summand. Clearly the basic modules $M$ which have all simples in $\mathcal{G}$ as summands are in 1-1 correspondence with collections of basic modules $M_1,\dots,M_t$ for the algebras $\Lambda_{n_1},\dots,\Lambda_{n_t}$, where $M_i$ has a copy of $S[n_i]$ as a summand if $\soc P_i\in \mathcal{G}$ and a copy of $S[1]$ as a summand if $\Top P_i\in\mathcal{G}$. Now Corollary~\ref{c:fb_nakayama} shows that $M$ is faithfully balanced if and only if the $M_i$ are faithfully balanced. The result thus follows from Lemma~\ref{l:nknformula}.
\end{proof}

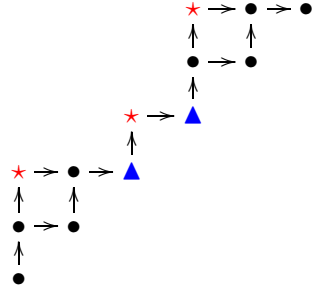
\begin{figure}[ht]
\[
\xymatrix@=0.3cm{
& & & \textcolor{red}{\star}\ar[r] & \bullet\ar[r] & \bullet \\
& & & \bullet\ar[r]\ar[u] & \bullet\ar[u] & \\
& & \textcolor{red}{\star}\ar[r] & \textcolor{blue}{\blacktriangle}\ar[u]\\
\textcolor{red}{\star}\ar[r]&\bullet\ar[r]&\textcolor{blue}{\blacktriangle}\ar[u]\\
\bullet\ar[r]\ar[u] & \bullet\ar[u]\\
\bullet\ar[u]
}
\]
\caption{Auslander-Reiten quiver of the algebra given by a quiver of type $A_6$ modulo the ideal generated by the paths from 2 to 4 and from 3 to 5. The red stars are the indecomposable projective-injective modules and the blue triangles are the simples in $\mathcal{G}$. The number of basic faithfully balanced modules
is $2^2 N_1(3) N_2(2) N_1(3) = 576$.}
\label{ex_nakayama}
\end{figure}

\section{Tree-like combinatorics for faithfully balanced modules}
\label{section_interleaved}
The purpose of the whole  section is to prove Theorem~\ref{t:fbncomb}. As explained in Theorem \ref{t:fban}, a basic faithfully balanced module can be identified with a collection of vertices in a staircase Young diagram. In order to be consistent with the literature on binary trees, we apply a rotation by an angle of $-\frac{\pi}{4}$ of the grid and we adopt the usual terminology of binary trees.

Let $M$ be a faithfully balanced module for $\Lambda_n$. The black box corresponding to the projective module $M_{1n}$ is at the top of the grid and is called the \emph{root}. The black boxes in the Young diagram are called \emph{vertices}. The first vertex in the cohook of a vertex $v$ which is on its right side (resp. its left side) is called, if it exists, the \emph{right parent} (resp. \emph{left parent}) of $v$. Conversely we say that $v$ is a \emph{right child} (resp. \emph{left child}) of its left parent (resp. right parent). The conditions $(FB1)$ and $(FB2)$ imply that every vertex (including the leaves) has at least one parent. 

We turn the collection of vertices into a graph in the Young diagram by adding a straight edge between each vertex (including the leaves) and each one of its parents.  If $M$ is a faithfully balanced module, we denote by $T_M$ the graph obtained as explained above and we call it the graph of $M$. From now on, we reserve the name \emph{vertex of $T_M$} for the vertices that are not the leaves. 

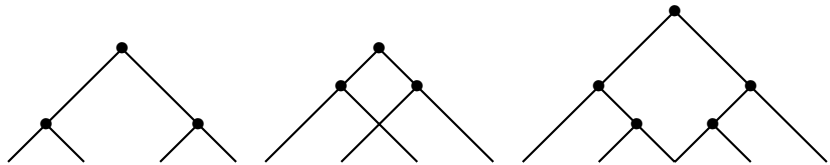
\begin{figure}[ht]
\centering
\begin{tikzpicture}
\node (1) at (0,0) {$\bullet$};
\node (2) at (1,-1) {$\bullet$};
\node (3) at (-1,-1) {$\bullet$};
\node (4) at (1.5,-1.5) {};
\node (5) at (0.5, -1.5) {};
\node (6) at (-0.5, -1.5) {};
\node (7) at (-1.5,-1.5) {} ; 
\draw[thick] (0,0)--(1,-1)--(1.5,-1.5)
			  (1,-1)--(0.5,-1.5)
			  (0,0)--(-1,-1)--(-0.5, -1.5)
			  (-1,-1)--(-1.5,-1.5);
\end{tikzpicture}
\begin{tikzpicture}
\node (1) at (0,0) {$\bullet$};
\node (2) at (0.5,-0.5) {$\bullet$};
\node (3) at (-0.5,-0.5) {$\bullet$};
\node (4) at (1.5,-1.5) {};
\node (5) at (0.5, -1.5) {};
\node (6) at (-0.5, -1.5) {};
\node (7) at (-1.5,-1.5) {} ; 
\draw[thick] (0,0)--(0.5,-0.5)--(1.5,-1.5)
			  (0.5,-0.5)--(-0.5, -1.5)
			  (0,0)--(-0.5,-0.5)--(-1.5,-1.5)
			  (-0.5,-0.5)--(0.5, -1.5);
\end{tikzpicture}
\begin{tikzpicture}[scale=2]
\node (1) at (0,-0.5) {$\bullet$};
\node (2) at (0.5,-1) {$\bullet$};
\node (3) at (0.25,-1.25) {$\bullet$};
\node (4) at (-0.5,-1) {$\bullet$};
\node (5) at (-0.25, -1.25) {$\bullet$};
\node (6) at (1, -1.5) {};
\node (7) at (0.5,-1.5) {} ;
\node (8) at (0,-1.5) {} ;
\node (9) at (-0.5,-1.5) {} ;
\node (10) at (-1,-1.5) {} ;
\draw[thick]  (0,-0.5)--(0.5,-1)--(1, -1.5)
			  (0.5,-1)--(0.25,-1.25)--(0.5,-1.5)
			  (0.25,-1.25)--(0,-1.5)
			   (0,-0.5)--(-0.5,-1)--(-1,-1.5)
			  (-0.5,-1)--(-0.25, -1.25)--(0,-1.5)
			  (-0.25, -1.25)--(-0.5,-1.5);

\end{tikzpicture}
\caption{The left-most and middle trees correspond to faithfully balanced modules for $\Lambda_3$. The right-most example is a minimal faithfully balanced module with $5$ summands for $\Lambda_4$.}\label{fig_ex_tree}
\end{figure}

\begin{lem}\label{lem_trees}
Let $n\in\mathbb{N}$, let $M$ be a minimal faithfully balanced module for $\Lambda_n$ and let $T_M$ be its graph. Then
\begin{enumerate}
    \item $T_M$ is a connected graph. 
    \item The number of vertices in $T_M$ is at least $n$. 
    \item $T_M$ is a rooted binary tree if and only if it has $n$ vertices.  
\end{enumerate}
\end{lem}

\begin{proof}
By using (FB1) and (FB2) we see that each vertex and each leaf is connected to the root of $T_M$.

If a vertex $v$ has a left and a right parent we can remove it without breaking the conditions (FB0), (FB1) and (FB2). As a consequence, in a minimal faithfully balanced module every vertex has one left parent or one right parent but not both. This implies that there is a unique path in $T_M$ between two (non-leaf) vertices. Moreover, all the vertices but the root are trivalent. 

The second point is proved by induction on $n$. For $n=0$ and $n=1$ there is nothing to prove. Assume $n\geqslant 2$. The root of $T_M$ has at least one child, say $S$. We consider the subgraph $T_S$ of $T_M$ that consists of all the vertices connected to $S$ in $T_M-\{R\}$ where $R$ is the root of $T_M$. Let $T_R$ be the graph obtained by cutting between $R$ and $S$ and removing all the vertices of $T_S$. We add a leaf at the former position of $S$ and we remove all the leaves which are no longer connected to $R$. 

Since there is a unique path between two vertices, the number of vertices of $T_M$ is equal to the sum of the numbers of vertices of $T_S$ and $T_R$. However, a leaf may appear in $T_S$ and in $T_R$ (see Figure \ref{fig_ex_tree} for an example). 

We denote by $n_S$ and $n_R$ the number of leaves of $T_S$ and $T_R$. The graphs $T_S$ and $T_R$ satisfy the conditions (FB0), (FB1) and (FB2) so they can be identified with graphs of faithfully balanced modules for $\Lambda_{n_S-1}$ and $\Lambda_{n_R-1}$ respectively. By induction, we see that the number of direct summands in $M$ is larger than $n_S+n_R-2$. At least one leaf occurs in both $T_S$ and $T_R$ (the one that we add at the former position of $S$ in $T_R$), so $n_S+n_R 
\geq n+1+1$ and the result follows. 

The last point is also easy to prove by induction. 

\end{proof}

\begin{rem}
The two trees on the left of Figure \ref{fig_ex_tree} give the same abstract graph but two different faithfully balanced modules. For us it is important to keep the `shape' of the tree. It can be done by considering it in the Young diagram, or alternatively by fixing the positions of the root and the leaves of the tree.   
\end{rem}

Recall that binary trees can be defined inductively as follows. A binary tree is either the empty set or a tuple $(r,L,R)$ where $r$ is a singleton set and $L$ and $R$ are two binary trees. The empty set has no vertex but has one leaf. The set of leaves of $T = (r,L,R)$ is the disjoint union of the set of leaves of $L$ and $R$. The \emph{size} of the tree is its number of vertices (equivalently the number of leaves minus $1$). As can be seen in Figure \ref{fig_ex_tree}, we draw the trees with their root on the top and the leaves on the bottom. We will always implicitly label the leaves of a tree of size $n$ from $1$ to $n+1$ starting from the right-most leaf. Let us give an inductive definition for the graphs $T_M$.

\begin{dfn}
\label{d:interleaved}
An \emph{interleaved tree} with $0$ vertices is the empty set. An interleaved tree with $n>0$ vertices is the data of
\begin{itemize}
\item A singleton set $r$ called the root.
\item Two interleaved trees $T_R$ and $T_L$ with, respectively $n_R$ and $n_L$ vertices such that $n = n_R+n_L + 1$,
\item A strictly increasing function $\operatorname{lea}_R : \{2,\dots, n_R +1\} \to \{2,\dots,n\}$. 
\end{itemize}
The function $\operatorname{lea}_R$ is called the \emph{interleaving} function.
\end{dfn}

\begin{rem}
Let $T$ be an interleaved tree. If the interleaving function satisfies $\operatorname{lea}_R(i)=i$ for all $i$ we say that it is a \emph{trivial} interleaving function and we say that $T$ has \emph{trivial interleaving}. The classical binary trees can be seen as interleaved trees which are inductively constructed from interleaved trees with trivial interleaving functions. 
\end{rem}

\begin{lem}\label{lem_inter}
Let $M$ be a faithfully balanced module with exactly $n$ summands for $\Lambda_n$. The graph $T_M$ can be naturally seen as an interleaved tree. 
\end{lem}

\begin{proof}
This graph has a root and a left and a right subtrees denoted by $T_L$ and $T_R$. The two subtrees correspond to faithfully balanced modules for $\Lambda_{n_L}$ and $\Lambda_{n_R}$ respectively. The function $\operatorname{lea}_R$ is defined by $\operatorname{lea}_R(i) = k $ if the $i$th leaf of $T_R$ is the $k$th leaf of $T$ for $i\in \{2,\dots,n_R\}$.
\end{proof}

The interleaving function $\operatorname{lea}_R$ determines another strictly increasing function $\operatorname{lea}_L : \{1,\dots,n_L\}\to \{2,\dots,n\}\backslash \mathrm{Im}(\operatorname{lea}_R)$. Note that the function $\operatorname{lea}_R$ is not defined at $1$. This is just for convenience: this function gives the positions of the leaves of the right subtree and the first leaf of the right subtree is always $1$. Similarly, the function $\operatorname{lea}_L$ is not defined at $n_L+1$ because the last leaf of the left subtree is always $n+1$.

\begin{pro}
\label{p:fbinter}
Let $n\in\mathbb{N}$. 
\begin{enumerate}
\item The map sending a faithfully balanced module $M$ for $\Lambda_n$ to the interleaved tree $T_M$ is a bijection between the set of isomorphism classes of basic faithfully balanced modules with exactly $n$-summands for $\Lambda_n$ and the set of interleaved trees with $n$ vertices. 
\item It restricts as a bijection between the set of isomorphism classes of basic tilting modules for $\Lambda_n$ and the set of binary trees with $n$ inner vertices.
\end{enumerate} 
\end{pro}

\begin{proof}
By Lemma \ref{lem_inter} the graph $T_M$ is an interleaved tree. Conversely, let $T = (r,T_R,T_L,\operatorname{lea}_R)$ be a interleaved tree with $n$-vertices. We can place it in the Young diagram of staircase shape as follows: 

The root is placed in the box with coordinate $(1,n)$. The leaves of $T_R$ are placed accordingly to the function $\operatorname{lea}_R$ and the leaves of $T_L$ are placed accordingly to the function $\operatorname{lea}_L$. The position of each vertex is determined by the positions of the leaves. Precisely, if $v$ is the root of the subtrees with right-most leaf $i_r$ and left-most leaf $i_l$, then it is in the box with coordinates $(i_r,i_l-1)$.

In other words, there is a bijection between interleaved trees and collections of $n$ vertices in the Young diagram of triangular shape that satisfies the conditions of Theorem \ref{t:fban}.

For the second point, we remark that a faithfully balanced module with $n$ summands is a tilting module if and only if it has no self extensions. It remains to see that there is an extension between two indecomposable modules if and only if there is a non-trivial interleaving in the corresponding tree. There is a non-trivial interleaving in $T$ if and only if there are two indecomposable summands $M_{ac}$ and $M_{bd}$ with the property that $a<b< c+1 \leq d$. By Lemma  8.1 of \cite{hille_vol_tilting}, $\Ext^1(M_{ac},M_{bd})\neq 0$ if and only if $a<b\leq c+1\leq d$. The case $b=c+1$ cannot appear since $T_M$ is a tree. We can also see that the bijection restricts as the one defined in Section $9$ of \cite{hille_vol_tilting}.
\end{proof}

Using this inductive definition we can construct a simple bijection between the set of interleaved binary trees and the set of increasing binary trees introduced by Fran\c con in Section $2$ of \cite{francon}. The bijection uses two intermediate functions that we call \emph{untangling} and \emph{reordering}. At the level of abstract trees the functions do nothing, but they will change the positions of the leaves, and so the interleaving of the trees. These changes will be encoded in a labelling of the vertices of the tree. 

We start by considering interleaved trees which are \emph{labelled} by integers. Let $T$ be an interleaved tree with $n$ vertices. A label of $T$ is a sequence of pairwise distinct integers $V = (v_1,v_2,\dots,v_n)$. The integer $v_i$ is the label of the $i$-th vertex in the pre-order traversal of $T$ (recursively visit the root, the right subtree and the left subtree). Then $v_1$ is the label of the root of $T$ and if $T_R$ has $n_R$ vertices, the sequence $(v_2,\dots,v_{n_R+1})$ labels the vertices of the subtree $T_R$. The remaining are the labels of the left subtree. Note that the ordering of the elements of the sequence is important!

\begin{dfn}
An increasing interleaved tree is an interleaved tree $T$ together with a labelling of its vertices by pairwise distinct integers such that if $v$ is a child of $w$, then the label of $w$ is smaller than the label of $v$. 
\end{dfn}

\begin{figure}[ht]
\centering 
\begin{tikzpicture}[scale =0.85]
\node (1) at (0,0) {$1$};
\node (2) at (0.5,-0.5) {$2$};
\node (3) at (2,-2) {$3$};
\node (4) at (3.5,-3.5) {$4$};
\node (5) at (3, -4) {$5$};
\node (6) at (-1.5, -2.5) {$6$};
\node (7) at (-1,-1) {$7$} ; 
\node (8) at (0,-2) {$8$};
\node (9) at (0.5,-2.5) {$9$};
\node (10) at (-1.5,-3.5) {$10$};
\node (a) at (5,-5) {}; \node (b) at (4,-5) {}; \node (c) at (3,-5) {}; \node (d) at (2,-5) {}; \node (e) at (1,-5) {}; \node (f) at (0,-5) {}; \node (g) at (-1,-5) {}; \node (h) at (-2,-5) {}; \node (i) at (-3,-5) {}; \node (j) at (-4,-5) {};\node (k) at (-5,-5) {};
\draw[red,thick] 
	  (2)--(3)--(4)--(a)
	  (4)--(5)--(d)
	  (5)--(b)
	  (3)--(g)
	  (2)--(6)--(j)
	  (6)--(e);
\draw[blue,thick] (7)--(8)--(9)--(c)
            (9)--(h)
		    (8)--(10)--(i)
		    (10)--(f)
		    (7)--(k);
\draw[thick] (1)--(7)
	  (1)--(2);
\end{tikzpicture}
\caption{An example of an increasing interleaved tree.}\label{fig_increasing}
\end{figure}
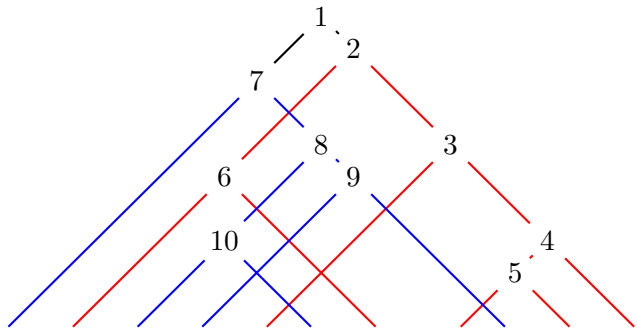
If $T$ is an interleaved tree of size $n$, we can always turn it into an increasing tree by associating to it the sequence of labels $(1,2,\dots,n)$. We say that the labelled interleaved tree $(T,V)$ is \emph{well-ordered} if the sequence $V$ is strictly increasing for the usual ordering of the integers. 

The first step of the bijection is given by the \emph{untangling} function that takes a \emph{well-ordered} increasing interleaved tree and gives an interleaved binary tree with \emph{trivial interleaving} function.

Let $\big(T=(r,T_R,T_L,\operatorname{lea}_R),V\big)$ be a well-ordered increasing interleaved tree. Let  $\operatorname{Unt}(T)=(r,T_R,T_L,\operatorname{triv})$ where $\operatorname{triv}$ is the trivial interleaving function. Let $\operatorname{Unt}(V)$ be the sequence consisting of the $v_i$s where $i$ runs \emph{first} through the positions of the leaves of the right subtree and \emph{then} through the positions of the leaves of the left subtree. In other words, $Unt(V)$ is the sequence obtained by concatenation of the sequences $(v_1)$, $(v_{\operatorname{lea}_R(i)})_{i\in \{2,\dots,n_R\}}$ and $(v_{\operatorname{lea}_L(i)})_{i\in \{1,\dots,n_L\}}$. The untangling function sends $(T,V)$ to $(\operatorname{Unt}(T),\operatorname{Unt(V)})$. See Figure \ref{fig_untangling} for an illustration. 

Conversely we define a \emph{reordering} function that takes an increasing interleaved tree with trivial interleaving function and well-ordered subtrees and produces a well-ordered interleaved tree.

Let $\big(T'=(r,T'_R,T'_L,\operatorname{triv}),V'\big)$ be an increasing interleaved tree. Let $\operatorname{Reo}(V')$ be the sequence putting the elements of $V'$ in a strictly increasing order. Let $\operatorname{Reo}(T') = (r,T'_L,T'_R,\operatorname{lea}'_R)$ be the interleaved binary tree where $\operatorname{lea}'_R(i)$ is defined as the position of $v_i$ in $\operatorname{Reo}(V')$. 

\begin{lem}\label{unt-reo}
The function $\operatorname{Unt}$ and $\operatorname{Reo}$ are mutually inverse bijections between the set of well-ordered interleaved binary trees and the set of increasing interleaved binary trees with trivial interleaving function and well-ordered subtrees. 
\end{lem}
\begin{proof}

Let $(T,V)$ be a well-ordered increasing tree with $n$ vertices and $(T',V')$ be an increasing interleaved tree with trivial interleaving function and well-ordered left and right subtrees. 

By construction $\operatorname{Unt}(T,V)$ is an interleaved tree with trivial function. Since the interleaving functions are strictly increasing we see that the subtrees $T_R$ and $T_L$ of $\operatorname{Unt}(T)$ are well-ordered.

 Conversely, the left and right subtrees of $T'$ are well-ordered so, the function $\operatorname{lea}'_R$ in $\operatorname{Reo}(T',V')$ is strictly increasing. So $\operatorname{Reo}(T')$ is an interleaved tree and it is by construction well-ordered.
 
 For $i \in \{1,2,\dots, n_R +1 \}$, the $i$-th element of $\operatorname{Unt}(V)$ is $v_{\operatorname{lea}_R(i)}$. Since $V$ is well-ordered, $v_{\operatorname{lea}_R(i)}$ is the $\operatorname{lea}_R(i)$-th largest element of $V$. It follows that $\operatorname{Reo}$ and $\operatorname{Unt}$ are two mutually inverse bijections.  
\end{proof}

We can now describe a bijection between the set of interleaved binary trees and the set of increasing binary trees. 

Starting with an interleaved tree with $n$ vertices, we see it as an increasing interleaved tree with label $V = (1,2,\dots, n)$. Applying the function $\operatorname{Unt}$ we obtain an increasing interleaved tree with a trivial interleaving function and well-ordered left and right subtrees. Then, we continue the process by inductively applying the untangling function to the left and right subtrees. Since at each step we go down in the tree, the process ends. Since we inductively remove the non trivial interleaving, the result is an increasing binary tree. We call this algorithm the \emph{untangling procedure}.

Conversely, starting with an increasing binary tree we inductively apply the function $\operatorname{Reo}$ to the subtrees of increasing size. The result is an interleaved tree labelled by $(1,2,\dots,n)$. We call this algorithm the \emph{reordering procedure}.

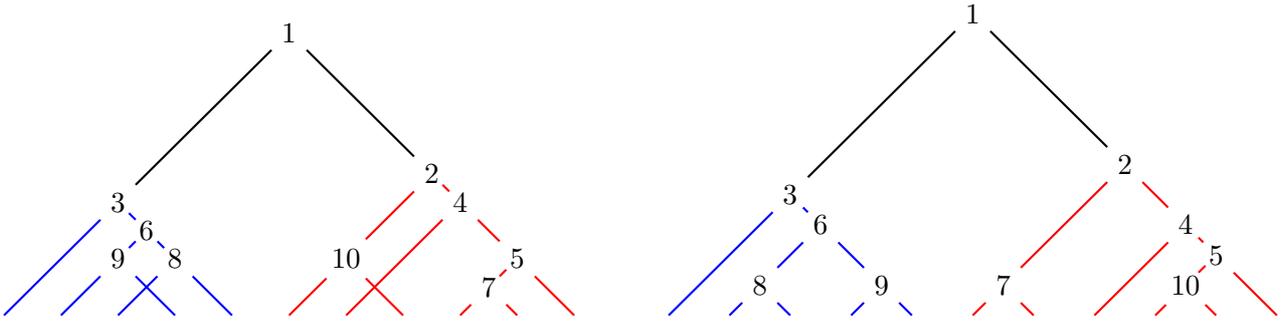
\begin{figure}[ht]
\begin{tikzpicture}[scale =0.750]
\node (1) at (0,0) {$1$};
\node (2) at (2.5,-2.5) {$2$};
\node (3) at (-3,-3) {$3$};
\node (4) at (3,-3) {$4$};
\node (5) at (4, -4) {$5$};
\node (6) at (-2.5, -3.5) {$6$};
\node (7) at (3.5,-4.5) {$7$} ; 
\node (8) at (-2,-4) {$8$};
\node (9) at (-3,-4) {$9$};
\node (10) at (1,-4) {$10$};
\node (a) at (5,-5) {}; \node (b) at (4,-5) {}; \node (c) at (3,-5) {}; \node (d) at (2,-5) {}; \node (e) at (1,-5) {}; \node (f) at (0,-5) {}; \node (g) at (-1,-5) {}; \node (h) at (-2,-5) {}; \node (i) at (-3,-5) {}; \node (j) at (-4,-5) {};\node (k) at (-5,-5) {};
\draw[red,thick] 
	  (2)--(4)--(5)--(5,-5)
	  (5)--(7)--(3,-5)
	  (7)--(4,-5)
	  (4)--(1,-5)
	  (2)--(10)--(0,-5)
	  (10)--(2,-5);
\draw[blue,thick] (3)--(6)--(8)--(-1,-5)
            (8)--(-3,-5)
		    (6)--(9)--(-4,-5)
		    (3)--(-5,-5)
		    (9)--(-2,-5);
\draw[thick] (1)--(3)
	  (1)--(2);
\end{tikzpicture}\hfill
\begin{tikzpicture}[scale =0.80]
\node (1) at (0,0) {$1$};
\node (2) at (2.5,-2.5) {$2$};
\node (3) at (-3,-3) {$3$};
\node (4) at (3.5,-3.5) {$4$};
\node (5) at (4, -4) {$5$};
\node (6) at (-2.5, -3.5) {$6$};
\node (7) at (3.5,-4.5) {$10$} ; 
\node (8) at (-1.5,-4.5) {$9$};
\node (9) at (-3.5,-4.5) {$8$};
\node (10) at (0.5,-4.5) {$7$};
\node (a) at (5,-5) {}; \node (b) at (4,-5) {}; \node (c) at (3,-5) {}; \node (d) at (2,-5) {}; \node (e) at (1,-5) {}; \node (f) at (0,-5) {}; \node (g) at (-1,-5) {}; \node (h) at (-2,-5) {}; \node (i) at (-3,-5) {}; \node (j) at (-4,-5) {};\node (k) at (-5,-5) {};
\draw[red,thick] 
	  (2)--(4)--(5)--(5,-5)
	  (5)--(7)--(3,-5)
	  (7)--(4,-5)
	  (4)--(2,-5)
	  (2)--(10)--(0,-5)
	  (10)--(1,-5);
\draw[blue,thick] (3)--(6)--(8)--(-1,-5)
            (8)--(-2,-5)
		    (6)--(9)--(-4,-5)
		    (3)--(-5,-5)
		    (9)--(-3,-5);
\draw[thick] (1)--(3)
	  (1)--(2);
\end{tikzpicture}
\caption{First and last steps of the untangling procedure applied to the example of Figure \ref{fig_increasing}.}\label{fig_untangling}
\end{figure}

\begin{pro}
\label{p:untang}
Let $n$ be an integer. 
\begin{enumerate}
    \item The untangling procedure induces a bijection between the set of interleaved trees with $n$ vertices and the set of increasing binary trees with $n$-vertices with inverse bijection given by the reordering procedure. 
    \item The map that sends an interleaved tree to the word obtained by reading in in-order (left subtree, root, right subtree) the label of the increasing binary tree given by the untangling procedure induces a bijection between the set of interleaved trees with $n$ vertices and the set of permutations on $\{1,2,\dots,n\}$. 
\end{enumerate}
\end{pro}
\begin{proof}
The first point follows from Lemma \ref{unt-reo}. It is classical that reading the labels of the vertices of an increasing tree in in-order induces a bijection between the set of increasing binary trees and the set of permutations (see Section $2$ of \cite{francon} for more details). 
\end{proof}
\begin{rem}
It is clear that the untangling procedure restricts as a bijection between binary trees and well-ordered increasing binary trees since all the untangling functions are the identity if we start with a binary tree. 
\end{rem}
The bijection between interleaved trees and increasing binary trees is natural, however the induced bijection with the set of permutations does not seem to reflect the interesting combinatorial properties that we observed in Corollary \ref{c:fbnumbers}. For that, we consider another classical family counted by $n!$. 

\begin{dfn}
\label{d:selfbded}
A function $f : \{1,2,\dots,n\} \to \{1,2,\dots,n\}$ is \emph{self-bounded} if $f(i)\leq i$ for $i \in \{1,2,\dots,n\}$. (These functions are called `d\'ecroissantes' by Fran\c con in \cite{francon}.)
\end{dfn}

The untangling procedure is also a way of labelling the vertices of an interleaved tree: if $T$ is an interleaved tree, the untangling procedure gives an increasing binary tree. Reading the tree using a traversal gives a sequence of labels that we use to label the vertices of $T$ using the same traversal. 

This labelling can be described as follows. Let $T$ be an interleaved tree with $n$ vertices with right subtree having $n_R$ vertices and left subtree having $n_L$ vertices. Let $V = (1,2,\dots,n)$. The label of the root of $T$ is the first element of $V$, which is $1$. Let $V_R = (v_{\operatorname{lea}_R(i)})_{i\in \{2,\dots,n_R\}}$ and $V_L = (v_{\operatorname{lea}_L(i)})_{i\in \{1,2,\dots,n_L\}}$. The sequence $V_R$ is associated to the right child of the root and $V_L$ to the left child. The label of a vertex is given by the first element of its associated sequence $V$. The interleaving function splits the sequence as a right and a left sequences that are respectively associated to the right and left child of the vertex. For the rest of the section, we assume that all interleaved trees are labelled in this way.

If $T$ is an interleaved tree we construct a function $f_T$ as follows. First label the vertices of $T$ by the procedure described above. If $v$ is a vertex labelled by $i$ we let $f_T(i) = j$ where $j$ is the position of the most right leaf of the subtree with root $i$ in $T$. 

In terms of faithfully balanced modules the function $f$ is obtained by taking the index of the simple top of each of the indecomposable summand of the module in a suitable total ordering of the indecomposable summands.  

For example, if $T$ is the interleaved tree of Figure \ref{fig_increasing}, then the function $f_T$ is:

\begin{equation}\label{ex_fun} \begin{tabular}{|c|llllllllll|}
	\hline
   i &  1 & 2 & 3 & 4 & 5 & 6 & 7 & 8 & 9 & 10 \\
  \hline
  $f_T(i)$ & 1 & 1 & 3 & 1 & 1 & 3 & 5 & 6 & 3 & 2 \\
  \hline
\end{tabular}
\end{equation}

Before proving that the function $f_T$ is self-bounded, we need a technical lemma. 

\begin{lem}\label{induc_tree}
Let $T$ be an interleaved tree with right subtree $T_R$ and left subtree $T_L$. Let $V_R=(v_1,\dots,v_{n_R})$ and $U_L = (u_1,\dots,u_{n_L})$ be the sequences of labels of the vertices of the subtrees $T_R$ and $T_L$ respectively. 
\begin{enumerate}
\item A vertex of the tree $T_R$ is labelled by the integer $i$ if and only if the corresponding vertex of $T$ is labelled by $v_i\in V_R$. Moreover, we have $f_{T_R}(i)=j$ if and only if $f_{T}(v_i) = v_{j-1}$ with the convention that $v_0 =1$. 
\item A vertex of the tree $T_L$ is labelled by the integer $i$ if and only if the corresponding vertex of $T$ is labelled by $u_i \in U_L$. Moreover, we have $f_{T_L}(i)=j$ if and only if $f_T(u_i) = u_j$. 
\end{enumerate} 
\end{lem}
\begin{proof}
This follows from the description of the labelling of the tree. Since the right child of the root of $T$ is labelled by the second leaf of the right subtree, there is a  shift in the description of $f_{T}$ in terms of $f_{T_R}$. On the other hand, the left child of the root of $T$ is labelled by the first leaf of $T_L$, so there is no shift.  
\end{proof}
\begin{lem}
Let $T$ be an interleaved tree. Then the function $f_T$ is self-bounded.
\end{lem}
\begin{proof}
We prove it by induction on the number of vertices of the trees. 

Let us denote by $V_R$ (resp. $U_L$) the labels of the right (resp. left) subtree of $T$. 

The root of $T$ is labelled by $1$, so $f_T(1)=1$. The root $r_R$ of the right subtree is labelled by the first element $v_R$ of $V_R$ and $f_T(v_R)=1<v_R$ and the root of the left subtree has for label the first leaf $v_L$ of the left subtree so $f_T(v_l)=v_L$.

If $v$ is a vertex of the right subtree, it corresponds to a vertex labelled by $i$ in the tree $T_R$ and by induction we have $f(i)\leq i$. By Lemma \ref{induc_tree} the vertex $v$ is labelled by the $i$-th element $v_i$ of $V_R$ and $f(v_i) = v_{j-1} < v_j \leq v_i$.  The proof is similar for the vertices of the left subtrees. 
\end{proof}

In order to show that the map $T\mapsto f_T$ is a bijection between interleaved trees and self-bounded functions, we show that the two sets have the same grammar. Using Lemma \ref{induc_tree} it is clear how to define a \emph{right} and a \emph{left sub-functions}.  We also need to define a partition of $\{2,\dots,n\}$ into two sequences $F_R$ and $F_L$. 

The sequence $F_R$ is the (totally ordered) sequence inductively constructed as follows. Let $1\neq i_1$ be the smallest integer such that $f(i_1) = 1$. If there is no such integer then $F_R$ is the empty sequence, otherwise $F_R = (i_1)$. For $i =i_1 +1 ,\dots, n$, if $f(i)=1$ or $f(i) \in F_R$, then add $i$ to $F_R$. The sequence $F_L$ is the sequence inductively constructed as follows. Let $i$ be the smallest integer such that $f(i)=i$ and for $i=2,\dots,n$ if $f(i)=j \in F_L$ or $f(i)=i$, then add $i$ to $F_L$. 

Looking at Lemma \ref{induc_tree} we see how to defined the left and right sub-functions of $f$. The right sub-function $f_R$ is defined by $f_R(i)=j$ if and only if $f(w_i) = w_{j-1}$ where $F_R = (w_1,\dots,w_{n_R})$ with the convention that $w_0 = 1$.  The left sub-function $f_L$ is defined by $f_L(i) = j$ if and only if $f(u_i)=u_j$ where $F_L = (u_1,\dots,u_{n_L})$.

In the case of the function (\ref{ex_fun}), we have $F_R = (2,4,5,7,10)$ and $F_L = (3,6,8,9)$. The sub-functions are
\begin{equation*}
\centering\begin{tabular}{|c|lllll|}
	\hline
   i &  1 & 2 & 3 & 4 & 5  \\
  \hline
  $f_R(i)$ & 1 & 1 & 1 & 4 & 2\\
  \hline
\end{tabular}
\ \ \ \ \ 
\begin{tabular}{|c|llll|}
	\hline
   i &  1 & 2 & 3 & 4  \\
  \hline
  $f_L(i)$ & 1 & 1 & 2 & 1 \\
  \hline
\end{tabular}
\end{equation*}

Using this decomposition of a self-bounded function we can inductively construct an interleaved tree: the root corresponds to $f(1)=1$, the interleaving function is defined by $\operatorname{lea}_r(i)$ is $i$th element of $F_R$. The right subtree corresponds to $f_R$ and the left subtree corresponds to $f_L$. The only function on the empty set corresponds to the trivial interleaved tree and the unique self-bounded function on a set with one element corresponds to the unique interleaved tree with one vertex.   

\begin{lem}\label{bij_lab}
Let $T$ be an interleaved tree with $n$ vertices and $f$ its self-bounded function. Then $\operatorname{Im}(\operatorname{lea}_R) = F_R$ and $\operatorname{Im}(\operatorname{lea}_L) = F_L$. 
\end{lem}

\begin{proof}
Let $i_1$ be the first element of $F_R$. Then $f(i_1)=1$, so $i_1$ is the label of the right child of the root of $T_R$. By construction of the labelling, it means that $i_1$ is the first element of $\operatorname{Im}(\operatorname{lea}_R)$. If $x$ is such that $f(x)=1$, then $x$ is in $T_R$. If $x$ is such that $f(x)= y$ with $y\in F_R$, then by induction $y$ is the label of a leaf of $T_R$. So the vertex labelled by $x$ is also in $T_R$. Since the vertices of $T_R$ are labelled by the leaves of $T_R$, we see that $F_R\subseteq \operatorname{Im}(\operatorname{lea}_R)$. 

Conversely if $x$ labels a leaf of $T_R$, then it labels a vertex of $T_R$ and $f(x)=1$ or $f(x)=y<x$ where $y$ labels a leaf of $T_R$.  So $F_R = \operatorname{Im}(\operatorname{lea}_R)$. The proof of the other case is left to the reader.
\end{proof}

\begin{thm}
\label{t:intsb}
The map sending an interleaved tree $T$ to the function $f_T$ is a bijection between the set of interleaved trees with $n$ vertices and the set of self-bounded functions on $\{1,2,\dots,n\}$. 
\end{thm}

\begin{proof}
Using Lemmas \ref{induc_tree} and \ref{bij_lab}, the result follows by induction on $n$. 
\end{proof}

\begin{rem}
This bijection is not the composition of the untangling procedure and the bijection between increasing binary trees and self-bounded functions given in Section $4$ of \cite{francon}.
\end{rem}

In the classical case of binary trees, the bijection restricts as a bijection between the set of binary trees with $n$ vertices and the set of non-decreasing self-bounded functions on $\{1,2,\dots,n\}$. These functions are known to be counted by the Catalan numbers (See e.g. part (s) of Exercise 6.19 of \cite{stanely99}).

\begin{pro}
\label{p:binnd}
Let $T$ be an interleaved tree with $n$ vertices and $f_T$ its self-bounded function. Then $T$ is a binary tree if and only if $f_T$ is such that $f_T(1)\leq f_T(2)\leq \dots \leq f_T(n)$.
\end{pro}

\begin{proof}
If $T$ is a binary tree, its labelling is well-ordered, it follows that $f_T(i)\leq f_T(i+1)$. Conversely, if $f_{T}(1)=1\leq f_T(2)\leq \dots \leq f_T(n)$, then the sequence $F_R$ is of the form $(2,3,\dots,k)$ because if $y$ is the smallest integers which is not in this sequence, then $f_T(y)=y$. Since $y\leq f_T(y+1)$, the value of $f_T(y+1)$ is either $y$ or $y+1$. This implies that $y+1$ is also in $F_L$ and we see that $F_L = \{y,y+1,\dots,n\}$. So the interleaving function of $T$ is trivial and the left and right sub-functions both satisfy the non-decreasing property of the Lemma. The result follows by induction.  
\end{proof}

\begin{proof}[Proof of Theorem \ref{t:fbncomb}]
The bijection between (i) and (ii) is given by Proposition~\ref{p:fbinter}, between (ii) and (iii) by Proposition~\ref{p:untang}, and between (ii) and (iv) by Theorem~\ref{t:intsb} and Proposition~\ref{p:binnd}.
\end{proof}

For $M$ a faithfully balanced module for $\Lambda_n$ with $n$ summands, we define $\chi(M) = \sum_i n_i (i-1)\in \mathbb{N}$, where $n_i$ is the number of indecomposable summands of $M$ in row $i$ of the Young diagram, or equivalently with top $S[i]$, so $\Top M \cong \bigoplus_{i=1}^n S[i]^{n_i}$. See \cite{Foata} for the notion of a `mahonian statistic'.

\begin{pro}
The mapping $\chi:\fb(n)\to \N$ is a mahonian statistic, that is,
\[
\sum_{M\in \fb(n)} q^{\chi(M)} = [n]_q!
\]
\end{pro}

\begin{proof}
By Theorem~\ref{t:fbncomb}, the faithfully balanced modules $M$ with $n$ summands correspond to self-bounded functions $f$, and by the discussion after Definition~\ref{d:selfbded}, $\chi(M) = \sum_{i=1}^n (f(i)-1)$. Thus
\[
\sum_{M\in \fb(n)} q^{\chi(M)} = \sum_f q^{\sum_{i=1}^n (f(i)-1)} = \prod_{i=1}^n \biggl(\sum_{f(i)=1}^i q^{f(i)-1}\biggr) = [n]_q!.
\qedhere
\]
\end{proof}

Using Theorem \ref{t:fban} and Lemma \ref{lem_trees}, we see that a faithfully balanced module with exactly $n$ summands for $\Lambda_n$ corresponds to a data of vertices in the Young diagram of staircase shape satisfying the following two conditions

\begin{enumerate}
\item There is a vertex in the top left box of the diagram.
\item Each vertex or leaf has a vertex on its left in the same row or above it in the same column but not both. 
\end{enumerate}
This is very similar to the definition of \emph{tree-like} tableaux in the sense of \cite{tree-like}. If there is an empty row or an empty column in the faithfully balanced module $M$, we can simply remove it and shrink the diagram. We denote the result by $sh(M)$. 

\begin{pro}\label{shrinking}
The map sending $M$ to $sh(M)$ is a bijection between the set of faithfully balanced modules with exactly $n$ summands and the tree-like tableaux with $n$ pointed cells. 
\end{pro}
\begin{proof}
Since we will not need this bijection, we only sketch the proof. We label the leaves of the grid $\Lambda_n$ from $1$ to $n+1$ starting at the top right and finishing at the bottom left. The southwest border of a tree-like tableaux can be seen as a path formed by vertical and horizontal steps. It has exactly $n+1$ steps that we label from $1$ to $n+1$ starting at the top right and finishing at the bottom left. In both cases, the labelling induces a labelling of the rows and the columns of the diagram. The vertex at the intersection of the row $i$ and the column $j$ is said to have coordinates $(i,j)$.

Let $T$ be a tree-like tableau with $n$ pointed cells. We can construct a configuration of vertices in the Young tableau $(n,n-1,\dots,1)$ by sending the pointed cell with coordinates $(i,j)$ to the vertex with same coordinates in the Young tableaux of staircase shape. 

It is straightforward to check that the result is a faithfully balanced module and that this map is a bijection which is inverse to $M \mapsto sh(M)$. 
\end{proof}

\begin{rem}
Tree-like tableaux are known to be counted by $n!$, so this gives another easy bijective proof for the cardinality of $\fb (n)$. However, it is not completely obvious that there are $n$! tree-like tableaux with $n$ pointed cells. Proposition~\ref{shrinking} relates $\fb(n)$ with other fillings of Young tabeaux such as \emph{permutation tableaux} (see e.g. \cite{permutation-tableaux}) and \emph{alternating tableaux} (see e.g. \cite{alternative-tableaux}). 

Finally let us remark that there is a bijection $\Phi_2$ between tree-like tableaux and increasing binary trees that can be found in \cite{tree-like}. Composing it with the bijection of Proposition \ref{shrinking}, we have another bijection between the set $\fb(n)$ and the set of increasing binary trees with $n$ vertices. The two bijections give the same underlying tree but the labelling are quite different. \end{rem}

\section{On partial orders} \label{s:poset}
Let $\La$ be a finite-dimensional Nakayama algebra. We define a relation 
$\unlhd$ on minimal faithfully balanced modules by
\[  
N \unlhd M \ \Leftrightarrow \ \text{$\cogen (N) \subseteq \cogen (M)$ and $\gen (N) \supseteq \gen (M)$.}
\]
It is clearly reflexive and transitive, and by Theorem \ref{t:mfb-gen-cogen} it is also antisymmetric, so a partial order. The relation $\unlhd $ has a smallest element given by $\La $ and a largest element given by $\kdual \La $, so its (finite) Hasse diagram is connected. As before, our main interest is in its restriction to $\fb (n)$ for the algebra~$\Lambda_n$.

\begin{rem}\label{r:tilting}
If $\La$ is hereditary and $M,N$ are cotilting modules (implying  $\cogen (X)= \cogen^1(X)$, $\gen (X)= \gen_1(X)$ for $X=M,N$), then the following are equivalent: (a) $N \unlhd M$, (b) $\cogen (N) \subseteq \cogen (M)$, (c) $\gen (N) \supseteq \gen (M)$, and (d) $\Ext^1(N,M)=0$. This suggests many possible partial orders generalizing the usual partial order on tilting modules for a hereditary algebras (cf. \cite{HU}). For example we can consider the partial order $\le$ given by
\[  
N \leq M \ \Leftrightarrow \ \text{$\cogen^1 (N) \subseteq \cogen^1 (M)$ and $\gen_1 (N) \supseteq \gen_1 (M)$}.
\]
In Figure~\ref{fig_posetfb3} we show the Hasse diagrams for $\fb(3)$.
The poset induced by the relation $\unlhd$ seems to be the most interesting, since the two others do not give lattice structures on $\fb(n)$ when $n\geq 4$. 
\end{rem}

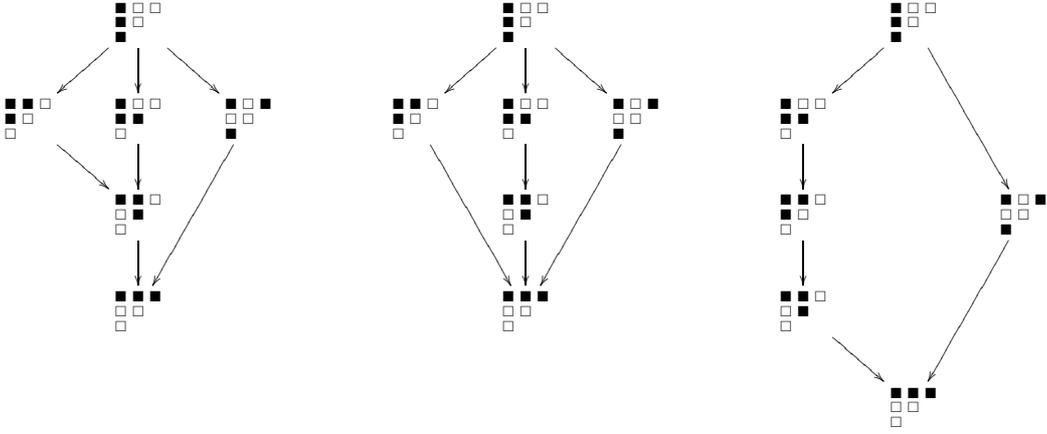
\begin{figure}[ht]
\[ 
\scalebox{0.7}{
\xymatrix{  &{\begin{smallmatrix}
\blacksquare & \square  & \square  \\[1pt]
\blacksquare  & \square\\[1pt]
\blacksquare 
\end{smallmatrix}} \ar[dr]\ar[dl]\ar[d] && && {\begin{smallmatrix}
\blacksquare & \square  & \square  \\[1pt]
\blacksquare  & \square\\[1pt]
\blacksquare 
\end{smallmatrix}}\ar[dr]\ar[dl]\ar[d] & && &{\begin{smallmatrix}
\blacksquare & \square  & \square  \\[1pt]
\blacksquare  & \square\\[1pt]
\blacksquare 
\end{smallmatrix}} \ar[ddr]\ar[dl] &\\
{\begin{smallmatrix}
\blacksquare & \blacksquare  & \square  \\[1pt]
\blacksquare  & \square\\[1pt]
\square 
\end{smallmatrix}}\ar[dr] & {\begin{smallmatrix}
\blacksquare & \square  & \square  \\[1pt]
\blacksquare  & \blacksquare\\[1pt]
\square 
\end{smallmatrix}}\ar[d] & {\begin{smallmatrix}
\blacksquare & \square  & \blacksquare  \\[1pt]
\square  & \square\\[1pt]
\blacksquare 
\end{smallmatrix}}\ar[ddl] & & {\begin{smallmatrix}
\blacksquare & \blacksquare  & \square  \\[1pt]
\blacksquare  & \square\\[1pt]
\square 
\end{smallmatrix}}\ar[ddr] & {\begin{smallmatrix}
\blacksquare & \square  & \square  \\[1pt]
\blacksquare  & \blacksquare\\[1pt]
\square 
\end{smallmatrix}} \ar[d] & {\begin{smallmatrix}
\blacksquare & \square  & \blacksquare  \\[1pt]
\square  & \square\\[1pt]
\blacksquare 
\end{smallmatrix}}\ar[ddl] && {\begin{smallmatrix}
\blacksquare & \square  & \square  \\[1pt]
\blacksquare  & \blacksquare\\[1pt]
\square 
\end{smallmatrix}} \ar[d] && \\
& {\begin{smallmatrix}
\blacksquare & \blacksquare  & \square  \\[1pt]
\square  & \blacksquare\\[1pt]
\square 
\end{smallmatrix}}\ar[d] && &&{\begin{smallmatrix}
\blacksquare & \blacksquare  & \square  \\[1pt]
\square  & \blacksquare\\[1pt]
\square 
\end{smallmatrix}}\ar[d] & &&  {\begin{smallmatrix}
\blacksquare & \blacksquare  & \square  \\[1pt]
\blacksquare  & \square\\[1pt]
\square 
\end{smallmatrix}}\ar[d] && {\begin{smallmatrix}
\blacksquare & \square  & \blacksquare  \\[1pt]
\square  & \square\\[1pt]
\blacksquare 
\end{smallmatrix}} \ar[ddl]\\
&{\begin{smallmatrix}
\blacksquare & \blacksquare  & \blacksquare  \\[1pt]
\square  & \square\\[1pt]
\square 
\end{smallmatrix}} && &&{\begin{smallmatrix}
\blacksquare & \blacksquare  & \blacksquare  \\[1pt]
\square  & \square\\[1pt]
\square 
\end{smallmatrix}} & && {\begin{smallmatrix}
\blacksquare & \blacksquare  & \square  \\[1pt]
\square  & \blacksquare\\[1pt]
\square 
\end{smallmatrix}}\ar[dr]&&\\
&&&& &&&& &{\begin{smallmatrix}
\blacksquare & \blacksquare  & \blacksquare  \\[1pt]
\square  & \square\\[1pt]
\square 
\end{smallmatrix}} &}}
\]
\caption{Hasse diagrams for $\fb(n)$ with respect to inclusion of $\cogen^1$-categories (left), $\leq$ (middle) and $\unlhd$ (right). The largest element is at the bottom.}
\label{fig_posetfb3} 
\end{figure}

\begin{dfn}
Using the canonical isomorphism of $K$-algebras $\varphi \colon \La_n^{op}\to \La_n$, the dual $\kdual M$ of any left $\Lambda_n$-module $M$ can be considered as a left $\Lambda_n$-module, which we denote by 
$M^\circ$. This defines a duality on the category $\lmod{\Lambda_n}$ which
preserves faithfully balancedness, so we have an involution 
$\fb (n) \to \fb (n)$ mapping $M \mapsto M^\circ$.
\end{dfn}

\begin{rem}
For any basic module $M$, the module $M^\circ$ can be found by reflecting the Auslander Reiten quiver along the symmetry axis passing through $M_{1,n}, M_{2, n-1}, M_{3, n-2}, \dots $ and these are the only indecomposable modules $X$ with $X^\circ\cong X$. In $\fb (3)$, we find two modules with $M^\circ \cong M$ and looking at our later example one can see $M \neq M^\circ$ for all $M \in \fb (4)$. We have $(\cogen^i (M))^\circ \cong \gen_i (M^\circ) $ for every $i \geq 0$. It is straightforward to see that
$M \unlhd N  \Leftrightarrow N^\circ \unlhd M^\circ$ for all $M,N \in \fb (n)$.
\end{rem}

Consider the poset $(\fb(n), \unlhd)$. A module $L$ is a common lower bound of $M$ and $N$ in $(\fb(n), \unlhd)$ if and only if $\cogen(L)\subseteq \cogen(M)\cap\cogen(N)$ and $\gen(L) \supseteq\gen(M)\cup\gen(N)$. For any two elements $M$ and $N$ in $(\fb(n), \unlhd)$ the module $\La_n$ is always a common lower bound of them. 

\begin{pro} 
\label{p:lattice}
The poset $(\fb(n), \unlhd)$ is a lattice for all $n\geq 1$.
\end{pro}

\begin{proof}
Since the poset is finite with a greatest element it is enough to show that it is a meet semi-lattice (see e.g. \cite[Proposition 3.3.1]{stanely12}). Thus we need to show that any two elements $M,N \in \fb(n)$ have a meet. 

By Lemma~\ref{l:cogenchar}, an indecomposable module is in $\cogen(M)\cap\cogen(N)$ if and only if it embeds in an indecomposable summand of $M$ and in an indecomposable summand of $N$. Let $C$ be the basic module such that $\add(C)$ is the minimal cocover of $\cogen(M)\cap \cogen(N)$, see \cite[\S 2]{ASm1}.
It follows that for each simple $S[i]$, there is at most one indecomposable summand of $C$ with socle $S[i]$, and if it occurs, it is a summand of $M$ or $N$.
Let $G$ be the basic module such that $\add(G)$ is the minimal cover of $\add(\gen(M)\cup \gen(N))$. 
Again, for each simple $S[i]$ there is at most one indecomposable summand of $G$ with top $S[i]$.
Clearly we have $\cogen(C)=\cogen(M)\cap \cogen(N)$ and $\gen(G)=\gen(\gen(M)\cup \gen(N))$. 
Moreover $C\in \gen(\gen(M)\cup \gen(N)) = \gen(G)$. 

Note that $|C|\leq n$ and equality holds if and only if $C=\kdual \La_n=M=N$. Namely, if $|C|=n$, then $M$ and $N$ have summands with socle $S[i]$, for all $i$. But since they are in $\fb(n)$, Theorem~\ref{t:fban} implies they are isomorphic to $\kdual \Lambda_n$. There is nothing to prove if $M=N$, so we may assume that $|C|=t<n$.
We write 
$G=G_1\oplus G_2\oplus \dots \oplus G_s$
where the tops of $G_1,\dots,G_s$ are $S[i_1],\dots,S[i_s]$ with $1\le i_1<i_2<\dots<i_s\leq n$.
Since $M_{1n}$ is a summand of $G$, we have $i_1=1$. For each $2\leq \alpha \leq s$, let $H_\alpha$ be  the indecomposable module with $\Top(H_{\alpha})=S[i_{\alpha}]$ and having the following properties:
\begin{itemize}
    \item[(P1)] $H_{\alpha}$ is a submodule of $C$,
    \item[(P2)] $G_{\alpha}$ is a quotient of $H_{\alpha}$,
    \item[(P3)] $H_{\alpha}$ has minimal length with respect to (P1) and (P2).
\end{itemize}
(Observe that the projective cover of $S[i_\alpha]$ satisfies (P1) and (P2) since it embeds in $M_{1,n}$, which is a direct summand of $C$.) Define $H=\bigoplus_{2\leq \alpha \leq s} H_{\alpha}$ and $L=C\oplus H$. We have $C\in \gen(G)\subseteq \gen(H\oplus M_{1n})$. 

We show that $L$ is faithfully balanced. The indecomposable direct summands of $H$ are submodules of $C$ by construction and since $C\in \gen(G) \subseteq \gen(L)$, we see that the module $L$ satisfies (FB0) and (FB1) in Theorem~\ref{t:fban}. Consider $\cohook(i,i-1)$ for $2\leq i \leq n$. If $S[i-1]\in \cogen(L)$, then we have $\cohook(i,i-1)\cap \add(L)\neq \emptyset$. If $S[i-1]\notin \cogen(L)=\cogen(C)$, then it is not in $\cogen(M)$ or $\cogen(N)$. Without loss of generality we may assume $S[i-1]\notin \cogen(M)$, then we must have $S[i]\in \gen(M)$ since $M$ is faithfully balanced. Thus we have $S[i]\in \gen(H)\subseteq\gen(L)$ and $\cohook(i,i-1)\cap \add(L)\neq \emptyset$. This proves that $L$ also satisfies (FB2), so it is a faithfully balanced module. 

Now we show that $|L|=n$, so $L\in\fb(n)$. The virtual cohook $(i,i-1)$ of $L$ is non-empty and it has three possible shapes accordingly to the following two conditions: $S[i-1]$ is or is not a submodule of $L$ and $S[i]$ is or is not a quotient of $L$. Let us denote by $u$ the number of virtual cohooks $(i,i-1)$ for which $S[i-1]$ is a submodule of $L$ and $S[i]$ is a quotient of $L$. Applying the inclusion-exclusion principle to virtual cohooks yields $0=(n-1)-(t-1)-(s-1)+u$. If $u\neq 0$, then by the construction of $H$ we know that there exists some $i$ such that $\cohook(i,i-1)$ contains an indecomposable summand of $C$ and an indecomposable summand of $G$. But this contradicts the fact that $M,N\in \fb(n)$. Hence we have $u=0$ and $|L|=t+(s-1)=n$, as desired.

By construction $\cogen(L)=\cogen(C)$ and $\gen(L)=\gen(H\oplus M_{1n})$, which implies that $L$ is a common lower bound of $M$ and $N$. 

Assume $L'\in \fb(n)$ is also a common lower bound of $M,N$. This means that $\cogen(L') \subseteq \cogen(M)\cap \cogen(N) = \cogen(L) = \cogen(C)$ and $\gen(G)=\gen(\gen(M)\cup\gen(N))\subseteq \gen(L')$. We have to show that $\gen(L)\subseteq \gen(L')$ to prove that $L' \unlhd L$. Since $L'$ is in $\cogen(L')\subseteq \cogen(C)$, we see that every indecomposable direct summand of $L'$ is a submodule of $C$. On the other hand, $G_{\alpha} \in \gen(G)\subseteq \gen(L')$. This means that there is an indecomposable direct summand $H'_{\alpha}$ of $L'$ such that $G_{\alpha}\in \gen(H'_{\alpha})$. By minimality of $H_{\alpha}$, we see that $H_{\alpha}\in \gen(H'_{\alpha})$. It follows that $\gen(L)\subseteq \gen(L')$.
\end{proof}

\begin{exa}
The following table gives two examples of the construction above for $n=4$. 
\[
\begin{tabular}{|c|c|c|c|c|c|}
\hline
$M$ & $N$ & $C$ & $G$ & $H$ & $L$ \\
\hline\\[-1em]
$\begin{smallmatrix}
\blacksquare & \square  & \square & \square  \\[1pt]
\blacksquare  & \blacksquare & \square  \\[1pt]
\blacksquare & \square  \\[1pt]
\square 
\end{smallmatrix}$  & $\begin{smallmatrix}
\blacksquare & \blacksquare  & \blacksquare & \square  \\[1pt]
\square  & \square & \blacksquare  \\[1pt]
\square & \square  \\[1pt]
\square 
\end{smallmatrix}$  & $\begin{smallmatrix}
\blacksquare & \square  & \square & \square  \\[1pt]
\square  & \blacksquare & \square  \\[1pt]
\square & \square  \\[1pt]
\square 
\end{smallmatrix}$ & $\begin{smallmatrix}
\blacksquare & \square  & \square & \square  \\[1pt]
\blacksquare & \square & \square  \\[1pt]
\blacksquare & \square  \\[1pt]
\square 
\end{smallmatrix}$ & $\begin{smallmatrix}
\square & \square  & \square & \square  \\[1pt]
\blacksquare & \square & \square  \\[1pt]
\blacksquare & \square  \\[1pt]
\square 
\end{smallmatrix}$ & $\begin{smallmatrix}
\blacksquare & \square  & \square & \square  \\[1pt]
\blacksquare  & \blacksquare & \square  \\[1pt]
\blacksquare & \square  \\[1pt]
\square 
\end{smallmatrix}$ \\[1em]
\hline\\[-1em]
$\begin{smallmatrix}
\blacksquare & \blacksquare  & \square & \blacksquare  \\[1pt]
\square  & \square & \square  \\[1pt]
\square & \blacksquare \\[1pt]
\square 
\end{smallmatrix}$  & $\begin{smallmatrix}
\blacksquare & \blacksquare  & \blacksquare & \square  \\[1pt]
\square  & \square & \blacksquare  \\[1pt]
\square & \square  \\[1pt]
\square 
\end{smallmatrix}$  & $\begin{smallmatrix}
\blacksquare & \blacksquare  & \square & \square  \\[1pt]
\square  & \square & \square  \\[1pt]
\square & \square  \\[1pt]
\square 
\end{smallmatrix}$ & $\begin{smallmatrix}
\blacksquare & \square  & \square & \square  \\[1pt]
\square & \square & \blacksquare  \\[1pt]
\square & \blacksquare  \\[1pt]
\square 
\end{smallmatrix}$ & $\begin{smallmatrix}
\square & \square  & \square & \square  \\[1pt]
\square & \blacksquare & \square  \\[1pt]
\square & \blacksquare  \\[1pt]
\square 
\end{smallmatrix}$ & $\begin{smallmatrix}
\blacksquare & \blacksquare  & \square & \square  \\[1pt]
\square  & \blacksquare & \square  \\[1pt]
\square & \blacksquare  \\[1pt]
\square 
\end{smallmatrix}$ \\[1em]
\hline
\end{tabular}
\label{meettable}
\]
Figure~\ref{fig_fb4} shows the Hasse diagram of $(\fb (4), \unlhd )$. The underlying graph of the Hasse diagram can be visualized as a truncated octahedron with two disected hexagons, as in Figure~\ref{fig_intro}.
\end{exa}

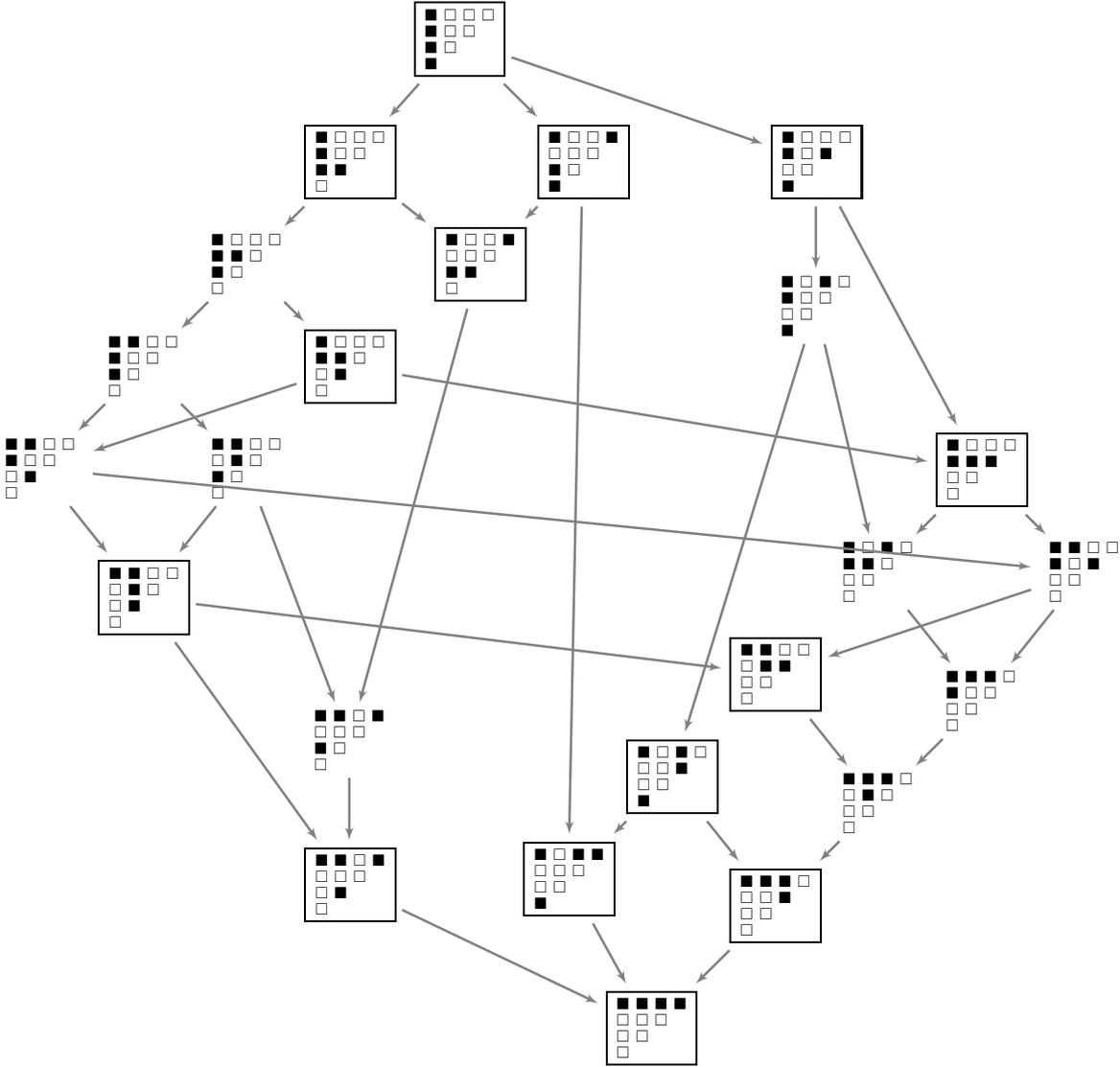
\begin{figure}[ht]
\begin{center}
\vskip 0.2cm
\scalebox{0.8}{
\tikzstyle{block} = [rectangle, fill=white!20,
    text width=4em, text centered, rounded corners, minimum height=1em, node distance=2.5cm]
\tikzstyle{line} = [draw, very thick, color=black!50, -latex']

\begin{tikzpicture}[scale=6, node distance = 1cm, auto]
    \node [block] (12) {\boxed{
\begin{smallmatrix}
\blacksquare & \square & \square & \square  \\[1pt]
\blacksquare  & \square & \square  \\[1pt]
\blacksquare & \square  \\[1pt]
\blacksquare 
\end{smallmatrix}
}};
\node [block, below right of=12, node distance=3cm] (24) {\boxed{\begin{smallmatrix}
\blacksquare & \square  & \square & \blacksquare  \\[1pt]
\square & \square & \square  \\[1pt]
\blacksquare & \square  \\[1pt]
\blacksquare 
\end{smallmatrix}}};

\node [block, left of=24, node distance=4cm] (11) {\boxed{\begin{smallmatrix}
\blacksquare & \square  & \square & \square  \\[1pt]
\blacksquare  & \square & \square  \\[1pt]
\blacksquare & \blacksquare  \\[1pt]
\square
\end{smallmatrix}}};

\node [block, right of=24, node distance=4cm] (10) {\boxed{\begin{smallmatrix}
\blacksquare & \square  & \square & \square  \\[1pt]
\blacksquare  & \square & \blacksquare  \\[1pt]
\square & \square  \\[1pt]
\blacksquare 
\end{smallmatrix}}};
 \node [block, below of=10] (9) {$\begin{smallmatrix}
\blacksquare & \square  & \blacksquare & \square  \\[1pt]
\blacksquare  & \square & \square  \\[1pt]
\square & \square  \\[1pt]
\blacksquare 
\end{smallmatrix}$};

 \node [block, below left of=11] (8) {$\begin{smallmatrix}
\blacksquare & \square  & \square & \square  \\[1pt]
\blacksquare  & \blacksquare & \square  \\[1pt]
\blacksquare & \square  \\[1pt]
\square
\end{smallmatrix}$};

\node [block, below left of=24 ] (23) {\boxed{
\begin{smallmatrix}
\blacksquare & \square  & \square & \blacksquare  \\[1pt]
\square  & \square & \square  \\[1pt]
\blacksquare & \blacksquare  \\[1pt]
\square 
\end{smallmatrix}
}};


 \node [block, below left of=8] (4) {$\begin{smallmatrix}
\blacksquare & \blacksquare  & \square & \square  \\[1pt]
\blacksquare  & \square & \square  \\[1pt]
\blacksquare & \square  \\[1pt]
\square
\end{smallmatrix}$};

 \node [block, below right of=8] (7) {\boxed{\begin{smallmatrix}
\blacksquare & \square  & \square & \square  \\[1pt]
\blacksquare  & \blacksquare & \square \\[1pt]
\square & \blacksquare  \\[1pt]
\square 
\end{smallmatrix}}};

\node [block, below left of=4] (3) {$\begin{smallmatrix}
\blacksquare & \blacksquare  & \square & \square  \\[1pt]
\blacksquare  & \square & \square  \\[1pt]
\square & \blacksquare  \\[1pt]
\square 
\end{smallmatrix}$};

\node [block, below right of=4] (16) {$\begin{smallmatrix}
\blacksquare & \blacksquare  & \square & \square  \\[1pt]
\square  & \blacksquare & \square  \\[1pt]
\blacksquare & \square  \\[1pt]
\square
\end{smallmatrix}$};

\node [block, below right of=9, node distance=4cm] (6) {\boxed{\begin{smallmatrix}
\blacksquare & \square  & \square & \square \\[1pt]
\blacksquare  & \blacksquare & \blacksquare  \\[1pt]
\square & \square  \\[1pt]
\square
\end{smallmatrix}}};

\node [block, below of=4, node distance=4cm] (15) {\boxed{\begin{smallmatrix}
\blacksquare & \blacksquare  & \square & \square  \\[1pt]
\square  & \blacksquare & \square  \\[1pt]
\square & \blacksquare  \\[1pt]
\square
\end{smallmatrix}}};

\node [block, below left of=6] (5) {$\begin{smallmatrix}
\blacksquare & \square  & \blacksquare & \square  \\[1pt]
\blacksquare  & \blacksquare & \square  \\[1pt]
\square & \square  \\[1pt]
\square 
\end{smallmatrix}$};

\node [block, below right of=6] (2) {$\begin{smallmatrix}
\blacksquare & \blacksquare  & \square & \square  \\[1pt]
\blacksquare  & \square & \blacksquare  \\[1pt]
\square & \square  \\[1pt]
\square
\end{smallmatrix}$};


\node [block, below of=6, node distance=4cm] (1) {$\begin{smallmatrix}
\blacksquare & \blacksquare  & \blacksquare & \square \\[1pt]
\blacksquare  & \square & \square  \\[1pt]
\square & \square  \\[1pt]
\square 
\end{smallmatrix}$}; 

 \node [block, below left of=5] (14) {\boxed{
\begin{smallmatrix}
\blacksquare & \blacksquare  & \square & \square  \\[1pt]
\square  & \blacksquare & \blacksquare  \\[1pt]
\square & \square  \\[1pt]
\square 
\end{smallmatrix}
}};

\node [block, below left of=14] (21) {\boxed{\begin{smallmatrix}
\blacksquare & \square & \blacksquare & \square \\[1pt]
\square  & \square & \blacksquare  \\[1pt]
\square & \square  \\[1pt]
\blacksquare
\end{smallmatrix}}};

\node [block, below left of=1] (13) {$\begin{smallmatrix}
\blacksquare & \blacksquare  & \blacksquare & \square  \\[1pt]
\square & \blacksquare & \square  \\[1pt]
\square & \square  \\[1pt]
\square
\end{smallmatrix}$};

\node [block, below of=11, node distance=10cm] (20) {$\begin{smallmatrix}
\blacksquare & \blacksquare  & \square & \blacksquare  \\[1pt]
\square  & \square & \square  \\[1pt]
\blacksquare & \square  \\[1pt]
\square
\end{smallmatrix}$};
 \node [block, below left of=21] (22) {\boxed{\begin{smallmatrix}
\blacksquare & \square  & \blacksquare & \blacksquare  \\[1pt]
\square  & \square & \square  \\[1pt]
\square & \square  \\[1pt]
\blacksquare 
\end{smallmatrix}}};

\node [block, below of=20] (19) {\boxed{\begin{smallmatrix}
\blacksquare & \blacksquare  & \square & \blacksquare  \\[1pt]
\square  & \square & \square  \\[1pt]
\square & \blacksquare  \\[1pt]
\square 
\end{smallmatrix}}};

\node [block, below left of=13] (17) {\boxed{\begin{smallmatrix}
\blacksquare & \blacksquare  & \blacksquare & \square  \\[1pt]
\square  & \square & \blacksquare  \\[1pt]
\square & \square  \\[1pt]
\square
\end{smallmatrix}}}; 
\node [block, below left of =17, node distance=3cm] (18) {\boxed{
\begin{smallmatrix}
\blacksquare & \blacksquare  & \blacksquare & \blacksquare  \\[1pt]
\square  & \square & \square  \\[1pt]
\square & \square \\[1pt]
\square
\end{smallmatrix}
}};

    \path [line] (12) -- (24);
    \path [line] (12) -- (10);
    \path [line] (12) -- (11);
    \path [line] (11) -- (8);
    \path [line] (24) -- (23);
    \path [line] (10) -- (9);
    \path [line] (11) -- (23);
    \path [line] (8) -- (4);
  \path [line] (24) -- (22);
  \path [line] (23) -- (20);
  \path [line] (9) -- (21);
  \path [line] (9) -- (5);
  \path [line] (10) -- (6);
   \path [line] (8) -- (7);
   \path [line] (4) -- (3);
   \path [line] (4) -- (16);
  \path [line] (7) -- (3);
   \path [line] (7) -- (6);
   \path [line] (3) -- (15);
  \path [line] (16) -- (15);
   \path [line] (16) -- (20);
   \path [line] (6) -- (5);
  \path [line] (6) -- (2);
   \path [line] (15) -- (19);
   \path [line] (15) -- (14);
  \path [line] (5) -- (1);
   \path [line] (14) -- (13);
   \path [line] (1) -- (13);
  \path [line] (20) -- (19);
   \path [line] (21) -- (22);
   \path [line] (21) -- (17);
  \path [line] (13) -- (17);
   \path [line] (19) -- (18);
   \path [line] (22) -- (18);
  \path [line] (17) -- (18);
  \path [line] (2) -- (1);
  \path [line] (2) -- (14);
  \path [line] (3) -- (2);
\end{tikzpicture}
}
\end{center}
\caption{Hasse diagram of $(\fb (4), \unlhd )$. The cotilting modules are in boxes.}
\label{fig_fb4}
\end{figure}

Given $M=X \oplus U \in \fb(n)$ with $X$ indecomposable, $X\neq M_{1n}$, we want to describe the possible indecomposable modules $Z$ such that $Z\oplus U \in \fb(n)$. 

Assume $X=M_{ij}$ is a \emph{splitting projective} module. If $\gen(X)\cap \add(U)\neq \{0\}$, we pick the unique $X_0=M_{it}\in \gen (X) \cap \add (U)$ of maximal length. If $\gen(X)\cap \add(U)= \{0\}$, we let $t = i-1$.
we define the \emph{internal cohook} as 
\[ 
\cohook_M(X) = \cohook(i,t)\cap (\gen(U) \cup \cogen (U)).
\] 
Assume $X=M_{ij}$ is a \emph{splitting injective} module. If $\cogen (X) \cap \add (U)\neq \{ 0\}$, we pick the unique $X_0=M_{vj}\in \cogen (X) \cap \add (U)$ of maximal length. If  $\cogen (X) \cap \add (U) = \{0\}$ we let $v = j+1$. We define the \emph{internal cohook} as 
\[ 
\cohook_M(X) = \cohook(v,j)\cap (\gen(U) \cup \cogen (U)). 
\]
Furthermore, we define a total order $\preceq$ on $\cohook_M(X)$ generated by the following covering relations: $A\preceq B$ if there is an irreducible map $A\to B$ or if $A$ and $B$ are both of minimal length in $\cohook_M(X)$ and $\Ext^1 (B,A)\neq 0$. This restricts to a total order on any subset. The module $X_0$ (or the leaf) is called the \emph{corner} of the internal cohook.

\begin{pro} \label{complements} Let $M=X\oplus U \in \fb (n)$ where $X$ is an indecomposable module, $X\neq M_{1n}$. For every indecomposable $Z$ the following are equivalent: 
\begin{itemize}
    \item[(1)] $Z \oplus U \in \fb (n)$ 
    \item[(2)] $Z\in \cohook_M(X)$. 
\end{itemize}
In particular, there is always an indecomposable injective $I$ and an indecomposable projective $P$ such that $I\oplus U, P\oplus U \in \fb(n)$. 
Assume now 
\[ 
\cohook_M(X) =\{P=Z_1\preceq Z_2 \preceq \dots \preceq Z_m =I\}
\]
then we have in $\fb (n)$
\[ Z_1 \oplus U \unlhd Z_2 \oplus U \unlhd \dots \unlhd Z_m \oplus U \] 
\end{pro}

\begin{proof}
The first part follows from Theorem \ref{t:fban}. The rest is clear from the definitions. 
\end{proof}

\begin{lem} 
Let $N,M\in \fb (n)$ and $N \unlhd M$. Then there is an $N'\in \fb(n)$ or an $M' \in \fb(n)$ such that 
\begin{itemize}
    \item[(a)] $N \unlhd N' \unlhd M$ and $\lvert \add(N) \cap \add (N')\rvert = n-1$, or 
    \item[(b)] $N \unlhd M' \unlhd M$ and $\lvert \add(M) \cap \add (M')\rvert = n-1$
\end{itemize}
is fulfilled. 
\end{lem}

\begin{proof}
The proof is purely combinatorial and requires a careful checking of the different possibilities. We write $M = Z \oplus V$ and $N = Z \oplus W$ where $\add(V)\cap \add(N) = \{0\}$ and $\add(W)\cap \add(M) = \{0\}$.

First assume that there is an indecomposable summand $X$ of $V$ that is \emph{splitting projective} in $\add(M)$. Since $X\in \gen(M) \subseteq \gen(N)$ there is an indecomposable summand $Y$ of $N$ such that $X\in \gen(Y)$. We choose it with minimal length with respect to this property and we let $U$ such that $M=U\oplus X$. Since $X$ is splitting projective, we see that $Y\in \add(W)$. Moreover, $Y\in \cogen(N)\subseteq \cogen(M)$. So $Y\in \cogen(U)$ and we see that
\begin{itemize}
\item $Y \in \cohook_M(X)$ and $Y\preceq X$. Proposition \ref{complements} tells us that $M' = U\oplus Y\in \fb(n)$.
\item By Proposition \ref{complements} we have $M'\unlhd M$.
\item $\cogen(M)=\cogen(M')$ because $X$ and $Y$ are in $\cogen(U)$.
\item $U\in \gen(M)\subset \gen(N)$ and $Y\in \gen(N)$.
\end{itemize}
The third points implies that $\cogen(N) \subseteq \cogen(M) = \cogen(M')$. The fourth point implies that $\gen(M')\subseteq \gen(N)$. In other words, we have $N\unlhd M'\unlhd M$.

If there is an indecomposable summand $Y$ of $W$ that is \emph{splitting injective} in $\add(N)$ we can construct $N'$ such that $N\unlhd N'\unlhd M$ by dualizing the previous argument. 

Now we assume that all the summands of $V$ are splitting injective in $\add(M)$ and all the summands of $W$ are splitting projective in $\add(N)$. We choose the indecomposable $X\in \add(V)$ maximal with respect to the index of its socle and then with respect to its top. As before, we write $M =U\oplus X$. Combinatorially, its column is the first (reading from left to right) that contains an element of $\add(W)$ and the module is the lowest element of $\add(W)$ in its column. In order to help the comprehension of the proof we draw in Figure \ref{fig_cohook} the shape of the Young diagram containing $M$ and $N$ at the neighborhood of $X$.
\begin{figure}[ht]
\centering 
\begin{tikzpicture}[scale =0.85]
\node (1) at (0,0) {$c$};
\node (2) at (0,2.5) {$\bullet$};
\node (3) at (0.75,2.5) {$X$};
\node (4) at (-1.5,2.5) {$\bullet$};
\node (5) at (-2,2.5) {$z$};
\draw (0,2.5)--(-1.5,2.5)
	  (0.5,0.5)--(-4,0.5)
	  (0.5,-0.5)--(-4,-0.5)
	  (-0.5,4)--(-0.5,-0.5)
	  (0.5,-0.5)--(0.5,4)      
      ;
\fill [pattern=horizontal lines] 
      (-0.5,2.75)--(0.5,2.75)--(0.5,4)--(-0.5,4)--cycle;
\fill [pattern=vertical lines] 
      (-0.5,0.5)--(-4,0.5)--(-4,-0.5)--(-0.5,-0.5)--cycle;
\end{tikzpicture}
\qquad \qquad
\begin{tikzpicture}[scale =0.85]
\node (1) at (0,0) {$c$};
\node (4) at (-1.5,2.5) {$\bullet$};
\node (5) at (-2,2.5) {$z$};
\node (6) at (-3,0) {$\bullet$};
\node (7) at (-2.5,0) {$Y$};
\node (8) at (-3,2) {$\bullet$};
\node (9) at (-2.5,2) {$w$};
\draw (0.5,0.5)--(-4,0.5)
	  (0.5,-0.5)--(-4,-0.5)
	  (-0.5,4)--(-0.5,-0.5)
	  (0.5,-0.5)--(0.5,4)    
	  (-3,0)--(-3,2)  
      ;
\fill [pattern=horizontal lines] 
      (-0.5,0.5)--(0.5,0.5)--(0.5,4)--(-0.5,4)--cycle;
\fill [pattern=vertical lines] 
      (-3.2,0.5)--(-4,0.5)--(-4,-0.5)--(-3.2,-0.5)--cycle;
\end{tikzpicture}
\caption{On the left the module $M$ on the right the module $N$.}\label{fig_cohook}
\end{figure}
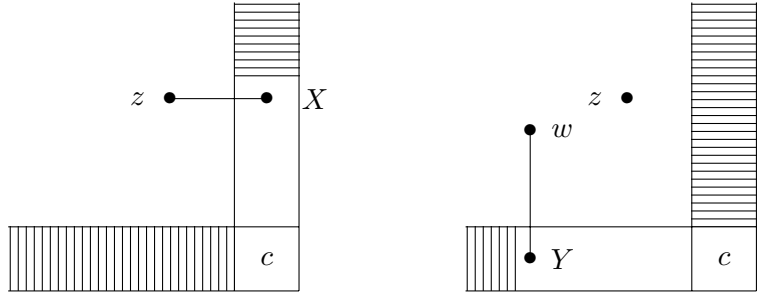
Let us start to explain the figure for the module $M$. Since $X$ is a splitting injective, there is $z\in \add(M)$ such that $X\in \gen(z)$ and there is no summands of $M$ in its column above it. We represent this by dashed horizontal lines. We denote by $c$ the corner of $\cohook_M(X)$. The hypothesis on $X$ implies that $z\in \add(Z)$ and that $c\in \add(Z)$ or is a leaf. Since $M$ is a faithfully balanced modules with exactly $n$ summands, there is no module in the same  row of $c$ on its left. We represent this by using dashed vertical lines. 

Let us move to the module $N$. By assumption, we have $\cogen(N)\subseteq \cogen(M)$. So, there is no indecomposable summand $I$ of $N$ such that $X\in \cogen(I)$. Combinatorial this means that the column of $c$ in $N$ is empty above (and equal to)  $X$. If there is an indecomposable summand $I$ of $N$ such that $c \subset I \subset X$, then by definition of $X$, this module is not a direct summand of $M$, so it is a splitting projective in $N$ module. But if $I \in \add (N)$ with $c\subsetneq I \subsetneq X$ is of maximal length it has to be a splitting injective. 
Since this is not possible, there is no module above $c$ in its column. Since $N$ is a faithfully balanced module, there is $Y\in \add(N)$ such that $c\in \gen(Y)$. In $Z$ there are no modules on the left of $c$, so $Y \in \add(W)$ and by assumption it is a splitting projective module in $\add(N)$. So there is $w\in \add(N)$ such that $Y$ is a proper submodule of $w$. Since $\cogen(N)\subseteq \cogen(M)$, we have $Y\in \cogen(U)$. In conclusion, we have:

\begin{itemize}
\item $Y\in \cohook_M(X)$, so $M' = U\oplus Y \in \fb(n)$ and $M'\unlhd M$.  
\item $U \in \gen(M) \subseteq \gen(N)$ and $Y\in \add(N)$, so $\gen(M')\subseteq \gen(N)$.
\item By construction if $I$ is an indecomposable module with $\soc(I)\neq \soc(c)$, then $I \in \cogen(M)$ if and only if $I \in \cogen(M')$. It follows that $\cogen(N)\subseteq \cogen(M')$.
\end{itemize}
It follows that $N\unlhd M'\unlhd M$. 
\end{proof}
\begin{cor} \label{c:covering-relation} 
If $N, M \in \fb (n)$ are neighbours in the Hasse diagram of $\unlhd $, then $\lvert \add (N) \cap \add (M)\rvert = n-1$.  
\end{cor}

Recall, whenever $N \unlhd M$ is a covering relation, we draw an arrow $N \to M$ in the Hasse diagram. 
\begin{cor}\label{c:cover_des}
Let $M\in \fb (n)$.
\begin{enumerate}
\item Let $X\in \add(M)$ be an indecomposable module. Assume that $X = Z_i$ in its internal cohook and we write $M=U\oplus X$. Then there is a cover relation $U\oplus X\unlhd U\oplus Y$ in $(\fb(n),\unlhd)$ if and only 
\begin{itemize}
\item $X$ is not injective.
\item $Z_{i+1}$, the successor of $X$ in its internal cohook, is not in $\add(M)$.
\item $Y$ = $Z_{i+1}$. 
\end{itemize}
\item In the Hasse diagram of $\unlhd$ we have:\\
The number of incoming arrows to (resp. outgoing from) $M$ is smaller or equal to the number of non-projective (resp. non-injective) indecomposable summands of $M$. 
\end{enumerate}

\end{cor}
\begin{proof}
If $M\unlhd N$ is a cover relation, then the two modules differ by exactly one indecomposable summand. Say that $M = U\oplus X$ and $M = U\oplus Y$. By Proposition \ref{complements}, we see that $Y$ must be in the internal cohook of $X$. Say that $X = Z_i$ in this cohook. Then $Y = Z_{j}$ for $i<j$. Because it is a cover relation, $j$ is the smallest integer such that $Z_j\notin \add(U)$. If $Z_{i+1} \in \add(U)$, then we write $M = V \oplus X \oplus Z_{i+1}.$ Then $M\unlhd N$ is not a cover relation because it factorizes as $M\unlhd V \oplus X \oplus Z_{j} \unlhd V\oplus Z_{i+1} \oplus Z_j = N$.
\end{proof}

\begin{proof}[Proof of Theorem \ref{thm-lattice}]
The first point is proved in Proposition \ref{p:lattice}. The third point is proved in Corollary \ref{c:cover_des}. The Tamari lattice is isomorphic to the poset of basic cotilting modules for $\Lambda_n$ with partial order given by $T_1\leq T_2$ if and only if $\Ext^1(T_1,T_2)=0$. Looking at Remark \ref{r:tilting} we see that it is a subposet of $(\fb(n),\unlhd)$. Recall that in the poset of cotilting modules the meet of two cotilting modules $T_1$ and $T_2$ is a cotilting module $T_3$ such that $\cogen(T_3)=\cogen(T_1)\cap \cogen(T_2)$ (See for example Section $11$ of \cite{thomas_tamari} for the dual statement on tilting modules). By construction, the meet $L$ of two faithfully balanced modules $M$ and $N$ is such that $\cogen(L)=\cogen(M)\cap \cogen(N)$, so in order to prove the second point it is enough to show that $L$ is a cotilting module when $M$ and $N$ are cotilting modules. Since $L$ has exactly $n$ non-isomorphic summands, it is enough to show that it has no self-extension. 

Let $M$ and $N$ be two cotilting modules and denote by $L$ their meet in $(\fb(n),\unlhd)$. Recall from the proof of Proposition \ref{p:lattice} that $L=C\oplus H$ where every indecomposable summand of $C$ is a direct summand of $M$ or $N$ and every indecomposable summand of $H$ is a submodule of an indecomposable summand of $C$. 

Assume that there are two indecomposable summands $M_{ac}$ and $M_{bd}$ of $L$ such that $\Ext^1(M_{ac},M_{bd})\neq 0$. By \cite[Lemma 8.1]{hille_vol_tilting} this is equivalent to $a<b\leq c+1\leq d$. If both $M_{ac}$ and $M_{bd}$ are summands of $C$, then we may assume that $M_{ac}\in \add(M)\setminus \add(N)$ and $M_{bd}\in \add(N)\setminus \add(M)$ since $M$ and $N$ have no self-extension. Now we have $M_{ac}\in \add(C)\subseteq \cogen(N)$. Thus there exists a summand $M_{a'c}$ of $N$ with $a'<a$ and so $\Ext^1(M_{a'c}, M_{bd})\neq 0$ which contradicts the fact that $N$ is cotilting. If $M_{ac}$ is a summand of $H$, then by construction it is a submodule of a summand $M_{ec}$ (so $e<a$) of $C$ and $\Ext^1(M_{ec}, M_{bd})\neq 0$. So we can assume that $M_{ac}$ is a summand of $C$ and $M_{bd}$ is a summand of $H$. By construction $M_{bd}$ surjects onto an indecomposable summand $M_{bg}$ of $G$, so a summand of $X=M$ or $N$. The minimality condition (P3) imposed on the summands of $H$ implies that $c<g$. This implies that $\Ext^1(M_{ac}, M_{bg})\neq 0$ which contradicts the fact that $X$ is a cotilting module. This proves that $L$ has no self-extension.
\end{proof}
\bibliographystyle{KLM13}
\bibliography{biblio}

\begin{bibdiv}
\begin{biblist}

\bib{AIR}{article}{
      author={Adachi, Takahide},
      author={Iyama, Osamu},
      author={Reiten, Idun},
       title={{$\tau$}-tilting theory},
        date={2014},
        ISSN={0010-437X},
     journal={Compos. Math.},
      volume={150},
      number={3},
       pages={415\ndash 452},
         url={https://doi.org/10.1112/S0010437X13007422},
      review={\MR{3187626}},
}

\bib{AF74}{book}{
      author={Anderson, Frank~W.},
      author={Fuller, Kent~R.},
       title={Rings and categories of modules},
   publisher={Springer-Verlag, New York-Heidelberg},
        date={1974},
        note={Graduate Texts in Mathematics, Vol. 13},
      review={\MR{0417223}},
}

\bib{ASS}{book}{
      author={Assem, Ibrahim},
      author={Simson, Daniel},
      author={Skowro\'{n}ski, Andrzej},
       title={Elements of the representation theory of associative algebras.
  {V}ol. 1},
      series={London Mathematical Society Student Texts},
   publisher={Cambridge University Press, Cambridge},
        date={2006},
      volume={65},
        ISBN={978-0-521-58423-4; 978-0-521-58631-3; 0-521-58631-3},
         url={https://doi.org/10.1017/CBO9780511614309},
        note={Techniques of representation theory},
      review={\MR{2197389}},
}

\bib{ASm1}{article}{
      author={Auslander, M.},
      author={Smal\o, Sverre~O.},
       title={Preprojective modules over {A}rtin algebras},
        date={1980},
        ISSN={0021-8693},
     journal={J. Algebra},
      volume={66},
      number={1},
       pages={61\ndash 122},
         url={https://doi.org/10.1016/0021-8693(80)90113-1},
      review={\MR{591246}},
}

\bib{ARS}{book}{
      author={Auslander, Maurice},
      author={Reiten, Idun},
      author={Smal\o, Sverre~O.},
       title={Representation theory of {A}rtin algebras},
      series={Cambridge Studies in Advanced Mathematics},
   publisher={Cambridge University Press, Cambridge},
        date={1995},
      volume={36},
        ISBN={0-521-41134-3},
         url={https://doi.org/10.1017/CBO9780511623608},
      review={\MR{1314422}},
}

\bib{tree-like}{article}{
      author={Aval, Jean-Christophe},
      author={Boussicault, Adrien},
      author={Nadeau, Philippe},
       title={Tree-like tableaux},
        date={2013},
        ISSN={1077-8926},
     journal={Electron. J. Combin.},
      volume={20},
      number={4},
       pages={Paper 34, 24},
      review={\MR{3158273}},
}

\bib{Foata}{inproceedings}{
      author={Foata, Dominique},
       title={Distributions eul\'{e}riennes et mahoniennes sur le groupe des
  permutations},
        date={1977},
   booktitle={Higher combinatorics ({P}roc. {NATO} {A}dvanced {S}tudy {I}nst.,
  {B}erlin, 1976)},
      series={NATO Adv. Study Inst. Ser., Ser. C: Math. Phys. Sci.},
      volume={31},
   publisher={Reidel, Dordrecht-Boston, Mass.},
       pages={27\ndash 49},
        note={With a comment by Richard P. Stanley},
      review={\MR{519777}},
}

\bib{francon}{article}{
      author={Fran\c{c}on, J.},
       title={Arbres binaires de recherche : propri\'et\'es combinatoires et
  applications},
        date={1976},
     journal={RAIRO - Theoretical Informatics and Applications - Informatique
  Th\'eorique et Applications},
      volume={10(R3)},
       pages={35\ndash 51},
}

\bib{Gab81}{article}{
      author={Gabriel, P.},
       title={Un jeu? les nombres de catalan},
        date={1981},
     journal={Uni Z\"{u}rich, Mitteilungsblatt des Rektorats, 12. Jahrgang},
      volume={Heft 6},
       pages={4\ndash 5},
}

\bib{HU}{article}{
      author={Happel, Dieter},
      author={Unger, Luise},
       title={On a partial order of tilting modules},
        date={2005},
        ISSN={1386-923X},
     journal={Algebr. Represent. Theory},
      volume={8},
      number={2},
       pages={147\ndash 156},
         url={https://doi.org/10.1007/s10468-005-3595-2},
      review={\MR{2162278}},
}

\bib{hille_vol_tilting}{article}{
      author={Hille, L.},
       title={On the volume of a tilting module},
        date={2006},
        ISSN={0025-5858},
     journal={Abh. Math. Sem. Univ. Hamburg},
      volume={76},
       pages={261\ndash 277},
         url={https://doi.org/10.1007/BF02960868},
      review={\MR{2293445}},
}

\bib{KSX01}{article}{
      author={K\"{o}nig, Steffen},
      author={Slung{\aa}rd, Inger~Heidi},
      author={Xi, Changchang},
       title={Double centralizer properties, dominant dimension, and tilting
  modules},
        date={2001},
        ISSN={0021-8693},
     journal={J. Algebra},
      volume={240},
      number={1},
       pages={393\ndash 412},
         url={https://doi.org/10.1006/jabr.2000.8726},
      review={\MR{1830559}},
}

\bib{MS19}{article}{
      author={{Ma}, Biao},
      author={{Sauter}, Julia},
       title={{On faithfully balanced modules, F-cotilting and F-Auslander
  algebras}},
        date={2019-01},
     journal={arXiv e-prints},
       pages={arXiv:1901.07855},
      eprint={1901.07855},
}

\bib{Morita}{article}{
      author={Morita, Kiiti},
       title={On algebras for which every faithful representation is its own
  second commutator},
        date={1958},
        ISSN={0025-5874},
     journal={Math. Z.},
      volume={69},
       pages={429\ndash 434},
         url={https://doi.org/10.1007/BF01187420},
      review={\MR{0126492}},
}

\bib{alternative-tableaux}{article}{
      author={Nadeau, Philippe},
       title={The structure of alternative tableaux},
        date={2011},
        ISSN={0097-3165},
     journal={J. Combin. Theory Ser. A},
      volume={118},
      number={5},
       pages={1638\ndash 1660},
         url={https://doi.org/10.1016/j.jcta.2011.01.012},
      review={\MR{2771605}},
}

\bib{PS18}{article}{
      author={{Pressland}, Matthew},
      author={{Sauter}, Julia},
       title={{On quiver Grassmannians and orbit closures for gen-finite
  modules}},
        date={2018-02},
     journal={arXiv e-prints},
       pages={arXiv:1802.01848},
      eprint={1802.01848},
}

\bib{Rin16}{incollection}{
      author={Ringel, Claus~Michael},
       title={The {C}atalan combinatorics of the hereditary {A}rtin algebras},
        date={2016},
   booktitle={Recent developments in representation theory},
      series={Contemp. Math.},
      volume={673},
   publisher={Amer. Math. Soc., Providence, RI},
       pages={51\ndash 177},
         url={https://doi.org/10.1090/conm/673/13490},
      review={\MR{3546710}},
}

\bib{SY11}{book}{
      author={Skowro\'{n}ski, Andrzej},
      author={Yamagata, Kunio},
       title={Frobenius algebras. {I}},
      series={EMS Textbooks in Mathematics},
   publisher={European Mathematical Society (EMS), Z\"{u}rich},
        date={2011},
        ISBN={978-3-03719-102-6},
         url={https://doi.org/10.4171/102},
        note={Basic representation theory},
      review={\MR{2894798}},
}

\bib{stanely99}{book}{
      author={Stanley, Richard~P.},
       title={Enumerative combinatorics. {V}ol. 2},
      series={Cambridge Studies in Advanced Mathematics},
   publisher={Cambridge University Press, Cambridge},
        date={1999},
      volume={62},
        ISBN={0-521-56069-1; 0-521-78987-7},
         url={https://doi.org/10.1017/CBO9780511609589},
        note={With a foreword by Gian-Carlo Rota and appendix 1 by Sergey
  Fomin},
      review={\MR{1676282}},
}

\bib{stanely12}{book}{
      author={Stanley, Richard~P.},
       title={Enumerative combinatorics. {V}olume 1},
     edition={Second},
      series={Cambridge Studies in Advanced Mathematics},
   publisher={Cambridge University Press, Cambridge},
        date={2012},
      volume={49},
        ISBN={978-1-107-60262-5},
      review={\MR{2868112}},
}

\bib{permutation-tableaux}{article}{
      author={Steingr\'{\i}msson, Einar},
      author={Williams, Lauren~K.},
       title={Permutation tableaux and permutation patterns},
        date={2007},
        ISSN={0097-3165},
     journal={J. Combin. Theory Ser. A},
      volume={114},
      number={2},
       pages={211\ndash 234},
         url={https://doi.org/10.1016/j.jcta.2006.04.001},
      review={\MR{2293088}},
}

\bib{qpa}{manual}{
      author={{The QPA-team}},
       title={{QPA} - {Q}uivers, path algebras and representations - a {GAP}
  package, version 1.25},
        date={2016},
        note={{\tt https://folk.ntnu.no/oyvinso/QPA/}},
}

\bib{sage}{manual}{
      author={{The Sage Developers}},
       title={{S}agemath, the {S}age {M}athematics {S}oftware {S}ystem},
        date={2019},
        note={{\tt https://www.sagemath.org}},
}

\bib{thomas_tamari}{incollection}{
      author={Thomas, Hugh},
       title={The {T}amari lattice as it arises in quiver representations},
        date={2012},
   booktitle={Associahedra, {T}amari lattices and related structures},
      series={Prog. Math. Phys.},
      volume={299},
   publisher={Birkh\"{a}user/Springer, Basel},
       pages={281\ndash 291},
         url={https://doi.org/10.1007/978-3-0348-0405-9_14},
      review={\MR{3221543}},
}

\bib{Thrall48}{article}{
      author={Thrall, R.~M.},
       title={Some generalization of quasi-{F}robenius algebras},
        date={1948},
        ISSN={0002-9947},
     journal={Trans. Amer. Math. Soc.},
      volume={64},
       pages={173\ndash 183},
         url={https://doi.org/10.2307/1990561},
      review={\MR{0026048}},
}

\bib{W88}{article}{
      author={Wakamatsu, Takayoshi},
       title={On modules with trivial self-extensions},
        date={1988},
        ISSN={0021-8693},
     journal={J. Algebra},
      volume={114},
      number={1},
       pages={106\ndash 114},
         url={https://doi.org/10.1016/0021-8693(88)90215-3},
      review={\MR{931903}},
}

\bib{W90}{article}{
      author={Wakamatsu, Takayoshi},
       title={Stable equivalence for self-injective algebras and a
  generalization of tilting modules},
        date={1990},
        ISSN={0021-8693},
     journal={J. Algebra},
      volume={134},
      number={2},
       pages={298\ndash 325},
         url={https://doi.org/10.1016/0021-8693(90)90055-S},
      review={\MR{1074331}},
}

\end{biblist}
\end{bibdiv}
\end{document}